\documentclass[11pt]{article}
\RequirePackage{amsthm,amsmath,amsfonts,amssymb}
\RequirePackage[numbers]{natbib}

\usepackage{hyperref}
\usepackage{pdfsync}
\usepackage[latin1]{inputenc} 
\usepackage{graphicx,epsfig}%%%GRAPHICS
\usepackage{color}
\usepackage{subcaption}
\usepackage{stmaryrd}

\newtheorem{remark}{\bf Remark}[section]
\usepackage{makeidx}
\makeindex
\def\R{\mathbb{R}}

\def\N{\mathbb{N}}
\def\P{\mathbb{P}}
\def\E{\mathbb{E}}
\def\L{\mathbb{L}}

\def\R{\mathbb{R}}

\def\Z{\mathbb{Z}}

\def\1{\mbox{I\hspace{-.6em}1}} % fonction carct�ristique

\def\cov{\mbox{Cov}\,}
\def\var{\mbox{Var}\,}

\def\1{\mbox{\hspace{.2em}I\hspace{-.6em}1}} % fonction carct�ristique

\def\limiteasn{\renewcommand{\arraystretch}{0.5}
\begin{array}[t]{c}\stackrel{a.s.}{\longrightarrow} \\
{\scriptstyle
n\rightarrow \infty}\end{array}\renewcommand{\arraystretch}{1}}

\def\limiten{\renewcommand{\arraystretch}{0.5}
\begin{array}[t]{c}\stackrel{}{\longrightarrow} \\
{\scriptstyle
n\rightarrow \infty}\end{array}\renewcommand{\arraystretch}{1}}

\def\limiteloin{\renewcommand{\arraystretch}{0.5}
\begin{array}[t]{c}\stackrel{{\cal L}}{\longrightarrow} \\
{\scriptstyle
n\rightarrow \infty}\end{array}\renewcommand{\arraystretch}{1}}

\def\egaleloi{\renewcommand{\arraystretch}{0.5}
\begin{array}[t]{c}\stackrel{{\cal L}}{\sim } \\
{}\end{array}\renewcommand{\arraystretch}{1}} 

\def\limiteproban{\renewcommand{\arraystretch}{0.5}
\begin{array}[t]{c}\stackrel{{\P}}{\longrightarrow} \\
{\scriptstyle
n\rightarrow \infty}\end{array}\renewcommand{\arraystretch}{1}}

\def\limiteL1{\renewcommand{\arraystretch}{0.5}
\begin{array}[t]{c}\stackrel{{\L^1}}{\longrightarrow} \\
{\scriptstyle n\rightarrow+\infty}\end{array}\renewcommand{\arraystretch}{1}}

\theoremstyle{plain} 
\theoremstyle{remark}

\newcommand*\interior[1]{\overset{\mathsf{o}}{#1}}

  \newtheorem{prop}{Proposition}[section]
  \newtheorem{cor}{Corollary}[section]
\newtheorem{theo}{Theorem}[section]
 \newtheorem{lem}{Lemma}[section]

\begin{document}

\title{\bf A new estimator for LARCH processes}
%\runtitle{New estimator for LARCH processes}
\author{Jean-Marc Bardet, {\tt bardet@univ-paris1.fr}, \\ 
University Paris 1 Panth\'eon-Sorbonne, SAMM, France}
\maketitle
\begin{abstract}
The aim of this paper is to provide a new estimator of parameters for LARCH$(\infty)$ processes, and thus also for LARCH$(p)$ or GLARCH$(p,q)$ processes. This estimator results from minimising a contrast leading to a least squares estimator for the absolute values of the process. Strong consistency and asymptotic normality are shown, and convergence occurs at the rate $\sqrt n$ as well in short or long memory cases. Numerical experiments confirm the theoretical results and show that this new estimator significantly outperforms the smoothed quasi-maximum likelihood estimators or weighted least squares estimators commonly used for such processes.
\end{abstract}
\begin{quote}
\end{quote}

\section{Introduction} \label{Intro}
Linear AutoRegressive Conditional Heteroskedastic (LARCH) processes were introduced by Robinson (1991) to model the long-range dependence of volatility and leverage. They are studied for their stationarity and dependence properties in Robinson and Zaffaroni (1997), Giraitis {\it et al.} (2000), Berkes and Horv\'ath (2003), and Giraitis {\it et al.} (2004). A LARCH$(\infty)$ process $(X_t)_{t\in \Z}$ is defined by:
$$
X_t=\xi_t \, \Big(a_0+\sum_{j=1}^{\infty}a_j \, X_{t-j}\Big)\quad\mbox{for any $t\in \Z$},
$$
where $(\xi_t)_{t\in \Z}$ is a white noise. In what follows, we will consider a parametric class of such LARCH$(\infty)$ processes, {\it i.e.} $a_j=a_j(\theta)$ for any $j\in \N$ with $\theta \in \R^\ell$, where the functions $a_j(\cdot)$ are assumed to be known but the true parameter $\theta^*$ is unknown. This paper is devoted to studying the asymptotic properties of a new  estimator of $\theta^*$ (rather than a component of $\theta^*$, such as the location parameter, as was done in Beran, 2006) for LARCH$(\infty)$ processes when a trajectory $(X_1,\ldots,X_n)$ is observed (see more details in section \ref{Sec2}).  \\
LARCH$(\infty)$ processes, which conditionally represent heteroskedastic weak white noise, offer new perspectives for modeling financial data. This model has the advantage over GARCH formulations that volatility can be arbitrarily close to $0$, which is nevertheless the case for certain financial series (see, for example, the example on CAC40 index returns discussed in section \ref{Simu}). The square of a LARCH$(\infty)$ process can also exhibit the long memory property (see below), which is impossible for a stationary ARCH$(\infty)$ process (see more theoretical details in Giraitis {\it et al.}, 2009, and the same illustrative example in section \ref{Simu}).\\
~\\
For numerous affine causal processes, such as ARMA, GARCH, ARMA-GARCH, AR$(\infty)$, or ARCH$(\infty)$ processes, the Gaussian quasi-maximum likelihood (QML) provides a very accurate estimator (see further details in Bardet and Wintenberger, 2009). Even though a LARCH$(\infty)$ process or its special cases LARCH$(p)$ or GLARCH$(p,q)$ (see their definitions in \eqref{LARCHp} and \eqref{GLARCH}) are also causal affine time series, such a contrast cannot be used to estimate the parameter $\theta^*$. Indeed, the conditional variance of $X_t$ cannot be bounded near $0$, and this does not allow asymptotic results for such contrasts (see more details on this point in Beran and Sch\"utzner, 2009, Truquet, 2014, and in particular in Francq and Zako\"ian, 2010). Beran and Sch\"utzner (2009) and Truquet (2014) propose an interesting alternative estimator based on a family of smooth approximations of the QML estimate, and they establish the consistency and asymptotic normality of the estimator of $\theta^*$ in cases of short or long memory. Francq and Zako\"ian (2010) preferred to construct weighted least squares estimators, for which they also show consistency and asymptotic normality. Note that they also extend their results to AR$(p)$-LARCH$(q)$ processes as well as to Truquet. \\
~\\
We propose a new estimator obtained by minimizing a least-squares contrast of the absolute values of $(X_t)$ (see its precise definition in \eqref{thetaLAV}, section \ref{Nonstat}). Under assumptions that are not too restrictive, especially considering short- and long-term memory, strong consistency and asymptotic normality are established for this estimator. Moreover, only a fourth-order white noise moment is required for asymptotic normality (order $4$ in Truquet, 2014, and $5$ in Beran and Sch\"utzner, 2009, for the smoothed QML estimator, no order condition in Francq and Zako\"ian, 2010, for the weighted LS estimator when an appropriate weight family is chosen). A convergence rate $\sqrt n$ for this new estimator is proved for LARCH$(\infty)$ processes with short memory as well as for the smoothed QML estimator (Truquet, 2014) and for the weighted LS estimator (Francq and Zako\"ian, 2010), but with different asymptotic covariance matrices. This rate of convergence $\sqrt n$ is also established for the LARCH$(\infty)$ process with long memory, while for the smoothed QMLE in Beran and Sch\"utzner (2009) only a rate $n^{\beta}$ with $0 < \beta < 1/2$ is obtained. Note that such a rate of convergence is also obtained with QMLE for generalized quadratic ARCH processes in Grublyte {\it et al.} (2017). \\
Monte Carlo experiments confirm the asymptotic behavior of the estimator even for trajectories with not very large lengths. The performances of this new estimator are then compared with those obtained with the regularized QMLE (for which the choice of the regularization parameter is a real problem) and with those obtained with the weighted least squares estimator of Francq and Zako\"ian (2010). The results of these comparisons undoubtedly show the much faster convergence of this new estimator, especially compared to the smoothed QMLE. \\ 
~\\
The following section \ref{Sec2} is devoted to the definition and stationarity conditions of the considered LARCH$(\infty)$ processes. The main results concerning the definition and the asymptotic behavior of the new estimator are given in section \ref{Nonstat}. Numerical experiments are proposed in section \ref{Simu} and proofs are established in section \ref{Proofs}.
\section{LARCH$(\infty)$ processes} \label{Sec2}
For $\ell\in \N^*$ and $u=(u_i)_{1\leq i\leq \ell}\in \R^\ell$, denote $\displaystyle \|u \|=\sqrt{u^2_1+\cdots + u_\ell^2 }$ the usual Euclidian norm. More generally, for $k\in \N^*$, if $u=(u_{i_1,\ldots,i_k})_{1\leq i_1,\ldots,i_k \leq \ell} \in \R^{\ell^k}$, denote 
$ \|u \|=\sqrt{\sum_{1\leq i_1,\ldots,i_k \leq \ell}  u^2_{i_1,\ldots,i_k} }. $
Denote also $\|Z \|_p = \big (\E \big [ \|Z\|^p \big ] \big )^{1/p}$ for $p\geq 1$ where $Z$ is a random vector. For any $m\in \N^*$, any $\ell\in \N^*$ and $\psi:\theta \in \Theta\subset \R^\ell \mapsto \psi(\theta) \in \R$ such as  $\psi \in {\cal C}^m(\Theta)$, the space of $m$-times continuously differentiable functions on $\Theta$, denote for $1\leq k \leq m$:
\begin{multline*}
\partial _\theta \psi(\theta):= \Big ( \frac \partial {\partial \theta_i} \psi(\theta) \Big )_{1\leq i\leq \ell} \\
\mbox{and}\quad \partial^k _{\theta^k} \psi(\theta):= \Big ( \partial^k _{\theta_{i_1}\cdots  \theta_{i_k}} \psi(\theta) \Big )_{1\leq i_1,\ldots,i_k\leq \ell}=\Big ( \frac {\partial^k} {\partial \theta_{i_1}\cdots \partial \theta_{i_k}} \psi(\theta) \Big )_{1\leq i_1,\ldots,i_k\leq \ell}.
\end{multline*}
Here we study a LARCH$(\infty)$ process introduced in Robinson (1991) and also studied in Giraitis {\it et al.} (2000), Giraitis {\it et al.} (2004), Beran and Sch\"utzner (2009), Francq and Zako\"ian (2010) or Truquet (2014). For $r\geq 2$, we will consider the following assumption: \\
~\\
{\bf Assumption $A(r)$:} 
\begin{itemize}
\item $(\xi_t)_{t\in \Z}$ is a sequence of symmetric centered independent random variables with continuous distribution such as $\| \xi_0\|_1=1$ and $\| \xi_0 \|_r<\infty$;
\item For any $j\in \N$, $\theta \in \R^\ell \mapsto a_j(\theta) \in \R$ are known continuous functions and without loss of generality we will assume $a_0(\theta) > 0$ for any $\theta\in \R^\ell$. \\
\end{itemize}
We will define a LARCH$(\infty)$ process $(X_t)_{t\in \Z}$ using Assumption {\bf $A(r)$} with $r\geq 2$. Before this, define:
\begin{equation}\label{Theta}
\Theta (2)=\Big \{ \theta \in \R^\ell,~
\| \xi_0\|^2_2 \, \sum_{j=1}^{\infty} a^2_j(\theta) < 1 \Big \}.
\end{equation} 
Then, under Assumption {\bf $A(r)$} with $r\geq 2$, for $\theta^* \in \Theta$ where $\Theta$ is a compact subset of $\Theta(2)$, we define a LARCH$(\infty)$ process $(X_t)_{t\in \Z}$ by:
\begin{equation}\label{larch}
X_t=\xi_t \, \Big(a_0(\theta^*)+\sum_{j=1}^{\infty}a_j(\theta^*) \, X_{t-j}\Big)\quad\mbox{for any $t\in \Z$}.
\end{equation}
Under these conditions, Giraitis {\it et al.} (2004) have proved the stationarity of $(X_t)$ and the existence of $\|X_0\|_2$. From now on we will assume that $\theta^*$ is unknown 
\begin{remark}\label{Rem0}
The case $a_0(\theta)= 0$ implies $X_t=0$ a.s. for any $t\in \Z$, see Giraitis {\it et al.} (2000). This explains why, we assume $a_0(\theta) > 0$ for any $\theta\in \R^\ell$ in Assumption {\bf $A(r)$}. 
\end{remark}

\begin{remark}\label{Rem1}
Note that in Assumption {\bf $A(r)$} we assume $\| \xi_0\|_1=1$ and not $\| \xi_0\|_2=1$ as is usually done. This is explained by the expression of the estimator we will consider. However, the difference between these two normalization options is only a new parametrization, because with $\xi'_t=\xi_t /\|\xi_0\|_2$ for any $t \in \Z$ and using the linearity of its expression, \eqref{larch} could also be written as 
\begin{equation} \label{M2} 
X_t=\xi'_t \, \Big(a'_0(\theta^*)+\sum_{j=1}^{\infty}a'_j(\theta^*) \, X_{t-j}\Big)\quad\mbox{where $\|\xi'_0\|_2=1$ and $a'_j(\theta^*) =\|\xi_0\|_2 \, a_j(\theta^*)$ for any $j\in \N$}.
\end{equation} In the case of Gaussian white noise, for example, we have $\|\xi_0\|^2_2=\sigma_\xi^2=\pi/2$ when $\| \xi_0\|_1=1$.
\end{remark}
In the sequel, under Assumption {\bf $A(r)$} with $r\geq 4$, we will also consider $\| X_0 \|_4$ and for this we define 
\begin{equation}\label{Theta4}
\Theta(4)=\Big \{ \theta \in \R^\ell,~\|\xi_0 \|_4^4 \sum_{j=1}^{\infty} a^4_j(\theta)+ 6 \,
\| \xi_0\|^2_2 \, \sum_{j=1}^{\infty} a^2_j(\theta) < 1 \Big \}.
\end{equation} 
\begin{remark}\label{Rem00}
Unless we consider a particular distribution for the noise, for example and typically the Gaussian ${\cal N}\big (0, (\sqrt{\pi/2})^2 \big )$, the moments $\| \xi_0\|_2 \leq \|\xi_0 \|_4$ are unknown and can take any value greater than or equal to $1$ according to this distribution. So the sets $\Theta(2)$ and $\Theta(4)$ are indeed unknown if we consider only the hypothesis of a symmetric noise $(\xi_t)$ with a moment of order $4$ and satisfying $\E \big [|\xi_0|\big ]=1$. Note, however, that it is exactly the same when estimating the parameters of a GARCH$(p,q)$ or ARCH$(\infty)$ process by quasi-maximum likelihood, e.g. when we define the stationarity set from the Lyapunov exponents depending on the noise distribution, or when we consider a condition of $\| \xi_0\|_4$ for asymptotic normality. Numerically (see section \ref{Simu}), one will use domains of minimisation in $\theta$ much larger than $\Theta(2)$ or $\Theta(4)$, for example $0\leq a_1,\, a_2 \leq 1$ for a LARCH$(2)$ process, but the condition of belonging to $\Theta(2)$ will be able to be checked a posteriori from the value taken by the estimator. The fact that this condition is not verified when applying a long-memory LARCH$(\infty)$ model to financial data led us to choose another more complex long-memory LARCH$(\infty)$ model. 
\end{remark}
Three interesting special cases of LARCH$(\infty)$ processes can be mentioned: 
\begin{enumerate} 
\item A first special case of LARCH$(\infty)$ processes are the LARCH$(p)$ processes defined by:
\begin{equation}\label{LARCHp}
X_t=\xi_t \,\sigma_t\quad 
\mbox{with} \quad \sigma_t= a_{0}+ \sum_{i=1}^p a_{i}\, X_{t-i}, \quad \mbox{ for any } t \in \Z 
\end{equation}
Therefore, a LARCH$(p)$ process is a LARCH$(\infty)$ process defined in \eqref{larch} with $a_k(\theta)=a_k$ for $0\leq k \leq p$ and $\theta={}^t\big ( a_0,a_1,\ldots,a_p\big ) \in (0,\infty)\times \R^{p}$. For such a process, the sets $\Theta(2)$ defined in \eqref{Theta} and $\Theta(4)$ defined in \eqref{Theta4} become respectively:
\begin{eqnarray}
\nonumber \Theta_p(2)&=&\Big \{ \theta={}^t\big ( a_0,a_1,\ldots,a_p\big ) \in(0,\infty)\times \R^{p},~
\| \xi_0\|^2_2 \, \sum_{j=1}^{p} a^2_j < 1 \Big \}\quad \\
\label{Thetap} \mbox{and}~~ \Theta_p(4)&=&\Big \{ \theta={}^t\big ( a_0,a_1,\ldots,a_p\big ) \in(0,\infty)\times \R^{p},~\|\xi_0 \|_4^4 \sum_{j=1}^{p} a^4_j+ 6 \,
\| \xi_0\|^2_2 \, \sum_{j=1}^{p} a^2_j < 1 \Big \}. \quad
\end{eqnarray}
\item A natural extension of the LARCH$(p)$ processes under consideration are GLARCH$(p,q)$ processes, which follow the same procedure as the well-known transition from ARCH processes to GARCH processes. A GLARCH$(p,q)$ process is defined by
\begin{equation}\label{GLARCH}
X_t=\xi_t \,\sigma_t\quad 
\mbox{with} \quad \sigma_t= c_{0}+ \sum_{i=1}^p c_{i}\, X_{t-i}+\sum_{j=1}^q d_{j}\, \sigma_{t-j}, \quad \mbox{ for any } t \in \Z.
\end{equation}
To study such a process, one defines the polynomials $P(x)=1-\sum_{j=1}^q d_{j}\,x^j$ and $Q(x)=c_{0}+ \sum_{i=1}^p c_{i}\,x^i$. Then the previous iteration equation \eqref{GLARCH} is equivalent to $P(B)\, \sigma=Q(B)\, X$ where $B$ is the usual backward operator. In the following we assume that $P$ and $Q$ are coprime polynomials for $\theta=\theta^*$. \\
We define $\theta=\big (c_0,c_1,\ldots,d_1,\ldots,d_q \big ) '\in (0,\infty)\times \R^{p+q}$ and the coefficients $a_k(\theta)$ decrease exponentially decrease to $0$ when $k \to \infty$ (as it is usually known for ARMA$(p,q)$ processes since the roots of $P$ lie outside the unit circle). \\
For GLARCH$(p,q)$ process, the assumption for obtaining a stationary $2$nd-order solution of \eqref{GLARCH} is $\theta \in \Theta_{p,q}(2)$,  with 
\begin{equation}\label{Thetapq}
\Theta_{p,q}(2)=\Big \{ \theta \in (0,\infty)\times (-1,1)^{p+q},~\sum_{i=1}^{q} d^2_i+
\| \xi_0\|_2 \, \sum_{j=1}^{p} c^2_j <1 \Big \}.
\end{equation}
The calculation of $\Theta(4)$ for such GLARCH$(p,q)$ processes is not quite straightforward. In Giraitis {\it et al.} (2004) $\Theta(4)$ is simplified for GLARCH$(1,1)$ and it is established that:
$$
\Theta_{1,1}(4)=\Big \{ \theta={}^t(c_0,c_1,d_1)\in (0,\infty)\times \R^{2},~~\|\xi_0 \|_4^4 \, \frac {c^4_1}{1-d_1^4}+ 6 \,
\| \xi_0\|^2_2 \, \frac {c^2_1}{1-d_1^2} < 1 \Big \}.
$$
\item Another case we will study is that of LARCH$(\infty$) with long memory, {\it i.e.} such that there are $d(\theta)\in (0,1/2)$ and $L_\theta(\cdot)$ a slowly varying function such that:
\begin{equation}\label{LRD}
a_j(\theta)=L_\theta(j)\, j^{d(\theta)-1}\quad \mbox{for $j \in \N^*$}.
\end{equation}
This case was considered in particular in Robinson and Zaffaroni (1997) and Beran and Sch\"utzner (2009). In this paper, a parametric estimation procedure was studied for the case where $a_j(\theta)=c \, j^{d-1}$ for $j \in \N^*$ and $\theta={}^t(a_0,c,d)$ (see further details below). Note that in such a case $\sum_{j=1}^\infty |a_j(\theta)|=\infty$ but $\sum_{j=1}^\infty a^2_j(\theta)<\infty$.
\end{enumerate}
\begin{remark}
For this third type of example we focused on the long memory property and for that the parameter $d(\theta)$ (or simply the parameter $d$ in Beran and Sch\"utzner's example) is necessarily in $(0,1/2)$. However, nothing prevents to extend the definition of the process to $d(\theta)\in (-\infty,1/2)$ or to a compact set included in $(-\infty,1/2)$ when the estimator is applied. We would then lose the exclusive long memory character but we would gain in generality.
\end{remark}
\section{A new estimator of LARCH parameters} \label{Nonstat}
\subsection{Definition and consistency of the estimator} 
We consider here a special case of M-estimators for estimating $\theta^*$ from an observed trajectory $(X_1,\ldots,X_n)$ of a stationary solution of \eqref{larch}. For this purpose, let the following contrast function $\Phi$ for $x \in \R^\N$ and $\theta\in \R^\ell$ be defined by
\begin{equation}
\label{phiLAV} \Phi(x,\theta)=\Big (|x_1|-\big|a_0(\theta)+\sum_{j=1}^{\infty}a_j(\theta) \, x_{j+1}\big |\Big )^2.
\end{equation}
\begin{remark}
The choice of this contrast function $\Phi$ follows from the fact that under a classical identifiability assumption (see Assumption Id$(\Theta )$ below), $\theta^*$ is the unique minimum in $\Theta$ of $\E \big [\Phi\big ((X_{-k})_{k\geq 0},\theta\big )\big ]$ under the normalization condition $\|\xi_0\|_1=1$ (see the proof of Proposition \ref{consLARCH}, Part 3.). \\
Other contrast functions satisfy this property, such as $\Phi_4(x,\theta)=\Big ( x_1^2-\big (a_0(\theta)+\sum_{j=1}^\infty a_j(\theta) x_{j+1}\big )^2\Big )^2$ under the usual normalization condition $\|\xi_0\|_2=1$. Such M-estimators defined from $\Phi_4$ require moments of order 8 to preserve their asymptotic normality, while for $\Phi$ a moment of order 4 is sufficient. Note that weighted quadratic contrast $\Phi_{ FZ }$ derived from $\Phi_4$ (see its definition in \eqref{WLSE}) was defined in Francq and Zako\"ian (2010) and allows, with appropriate weights, to obtain the asymptotic normality without a moment condition (except $r\geq 1$). However, Monte Carlo experiments show that the convergence rate of the estimator defined by $\Phi$ is faster than that of the estimator defined by $\Phi_4$ or $\Phi_{ FZ }$. \\
Note, however, that contrast functions such as
$$
\Phi_2(x,\theta)=\Big ( x_1-\big (a_0(\theta)+\sum_{j=1}^\infty a_j(\theta) x_{j+1}\big )\Big )^2\quad \mbox{or}\quad \Phi_1(x,\theta)=\Big | x_1-\big (a_0(\theta)+\sum_{j=1}^\infty a_j(\theta) x_{j+1}\big )\Big |
$$ 
are not such that $\E \big [\Phi\big ((X_{-k})_{k\geq 0},\theta\big )\big ]$ has a unique minimum in $\theta^*$ on the set $\Theta$. For example, quick calculations show that:
\begin{multline*}
\E \big [\Phi_2\big ((X_{-k})_{k\geq 0},\theta\big )\big ]=\big (\|\xi_0\|_2^2-1\big )\, \E \Big [ \big (a_0(\theta^*)+\sum_{j=1}^\infty a_j(\theta^*) X_{-j}\big )^2\Big ] \\
+ \E \Big [\big (a_0(\theta)+\sum_{j=1}^\infty a_j(\theta) X_{-j}\big )^2- \big (a_0(\theta^*)+\sum_{j=1}^\infty a_j(\theta^*) X_{-j}\big )^2 \Big ],
\end{multline*}
which in general does not have a minimum at $\theta^*$. 
\end{remark}
\begin{remark}
Unlike, in particular, the papers by Francq and Zako\"ian (2010) or Truquet (2014), the case of an AR$(p')$-LARCH$(p)$ process or, more generally, an AR$(p')$-LARCH$(\infty)$ process is not treated here. Based on the $\Phi$ contrast chosen, it would also have been possible to extend the study to AR$(p')$-LARCH$(\infty)$ processes by considering the $\Phi'$ contrast such that:
$$
\Phi'(x,\theta')=\Big (\Big |x_1-\sum_{k=1}^{p'}b_k\,x_{1+k}\Big |-\Big|a_0(\theta)+\sum_{j=1}^{\infty}a_j(\theta) \, \Big |x_{j+1}-\sum_{k=1}^{p'}b_k\,x_{j+1+k}\Big |\Big )^2,
$$
where $\theta'={}^t\big (b_1,\ldots,b_{p'},{}^t \theta \big )$. The convergence and asymptotic normality proofs for LARCH$(\infty)$-processes proposed in the rest of the manuscript could then be extended to AR$(p')$-LARCH$(\infty)$ processes, but this would make them even more technical and difficult to follow.  
\end{remark}
Now we define the process $(\widetilde X)_{t\in \Z}$ by:
\begin{equation}
\label{Xtilde} 
\widetilde X_t=\Big \{ \begin{array}{ll}X_t& \mbox{for $t\geq 1$} \\
0& \mbox{for $t\leq 0$} \end{array} .
\end{equation}
From now on, let us consider $\Theta$ as a compact subset of $\Theta(2)$, implying that $(X_t)_{t\in\Z}$ is a stationary ergodic process satisfying $\|X_0\|_2<\infty$. Then define the following estimator:
\begin{equation}
\label{thetaLAV} \widehat \theta_n=\mbox{Arg}\! \min_{\! \! \! \!\! \! \theta \in \Theta} \, \frac1 n \, \sum_{t=1}^n \Phi\big ((\widetilde X_{t-k})_{k\geq 0} , \theta \big ).
\end{equation}
In the sequel, we add an assumption about the derivatives of functions $\theta \in \Theta \mapsto a_k(\theta)$ for $k\in \N$ (we use the convention $\partial^0_{\theta^0} a_k(\cdot)=a_k(\cdot)$). So for $i\in \N$ define:\\ 
~\\
{\bf Assumption $C_i(\Theta)$:} For any $k\in \N$, the functions $a_k \in {\cal C}^i(\Theta)$ and there exist $C_a > 0$ and $\overline d < 1/2$ satisfying 
\begin{equation}\label{CondAN}
\sup_{\theta \in \Theta}\Big \{\sum_{j=0}^i \big \|\partial^j_{\theta^j} a_k(\theta) \big \| \Big \} \leq C_a \, k^{\overline d-1} \quad \mbox{for any $k \in \N^*$.}
\end{equation}
As already mentioned in Beran and Sch\"utzner (2009), to prove the consistency and asymptotic normality of $\widehat \theta_n$, it is necessary to take the derivatives in $\theta$ of $\theta \in \Theta \mapsto M_\theta(t)$ for $t\in \Z$, where
\begin{equation} 
M_\theta(t) =a_0(\theta)+ \sum_{k=1}^\infty a_k(\theta)\, X_{t-k}\quad\mbox{for $t\in \Z$ and $\theta \in \Theta$.}
\end{equation}
Note that the convergence of this infinite sum in the $\L^2(\Omega)$-norm is done using a Volterra decomposition, as in Giraitis {\it et al.} (2000) or Giraitis {\it et al.} (2004). However, the existence of $M_\theta(t)$derivatives could be problematic, since the sequence $(a_k(\theta))$ is not summable in the case of a long memory. We therefore consider the assumption: \\
~\\
{\bf Assumption (S):} For every $t \in \Z$, $(M_\theta(t))_{\theta\in \Theta}$ is a separable stochastic process on $\Theta$. \\
~\\
Note that this assumption is not really restrictive, since a stochastic process can always be replaced by a separable version (see Remark 1 in Beran and Sch\"utzner, 2009). Moreover, as proved in Proposition 2 of Beran and Sch\"utzner (2009), Assumption (S) combined with Assumption $C_1(\Theta)$ (resp. $C_2(\Theta)$), where $\Theta$ is a compact set of the $\Theta(2)\subset \R^\ell$, implies that $\theta \in \Theta \mapsto M_{\theta}(t)$ is almost surely 
 is differentiable once (resp. twice). \\
We add a classical identification condition:\\
~\\
{\bf Assumption Id$(\Theta)$}: If $\theta, \,\widetilde \theta \in \Theta$,
 \begin{equation}\label{identLARCH}
\big (a_i(\theta)=a_i(\widetilde\theta)~ \mbox{for all $i \in \N$} \big )~\implies ~\big (\theta=\widetilde \theta \big).
\end{equation}
Then we obtain the following conditions for the consistency of the two estimators:
\begin{prop}\label{consLARCH}
Under Assumption {\bf $A(2)$}, if $\theta^*\in \Theta$ is an unknown parameter, where $\Theta$ is a compact subset of $\Theta(2)\subset \R^\ell$ defined in \eqref{Theta}, consider $(X_t)_{t \in \Z}$ as a stationary solution of \eqref{larch} and $(X_1,\ldots,X_n)$ an observed trajectory of $(X_t)$. Assume also Assumption $C_\ell(\Theta)$, (S) and Id$(\Theta)$. Then $\widehat \theta_n \limiteasn \theta^*$ where $\widehat \theta_n$ is defined in \eqref{thetaLAV}. 
\end{prop}
Proposition \ref{consLARCH} can be specified for the special cases considered earlier, starting with the LARCH$(p)$ processes.
\begin{cor}\label{consLARCHp}
Under Assumption $A(2)$, with $p\geq 1$, let  $(X_t)$ be a LARCH$(p)$ process, solution of 
\begin{equation}\label{LARCHp*}
X_t=\xi_t \,\sigma_t\quad 
\mbox{with} \quad \sigma_t= a^*_{0}+ \sum_{i=1}^p a^*_{i}\, X_{t-i}, \quad \mbox{ for any } t \in \Z,
\end{equation}
where $\theta^*={}^t\big ( a^*_0,a^*_1,\ldots,a^*_p\big )\in\Theta$, a compact subset of $\Theta_p(2)$ defined in \eqref{Thetap}. Let $(X_1,\ldots,X_n)$ be an observed trajectory of $(X_t)$.
Then $\widehat \theta_n={}^t\big ( \widehat a^{(n)}_0,\widehat a^{(n)}_1,\ldots,\widehat a^{(n)}_p\big ) \limiteasn \theta^*$ where $\widehat \theta_n$ is defined in \eqref{thetaLAV} 
\end{cor}
Note that since the functions $a_k$ considered in the Corollary \ref{consLARCHp} are constant functions, the Assumptions Id$(\Theta)$, $C_\ell(\Theta)$ and (S) are satisfied. The same is true if we consider the case of GLARCH$(p,q)$ processes:
\begin{cor}\label{consGLARCH}
Under Assumption $A(2)$, with $p\geq 1,~q\geq 1$, let  $(X_t)$ be a GLARCH$(p,q)$ process, solution of 
\begin{equation}\label{LARCHpq*}
X_t=\xi_t \,\sigma_t\quad 
\mbox{with} \quad \sigma_t= c^*_{0}+ \sum_{i=1}^p c^*_{i}\, X_{t-i}+\sum_{j=1}^q d^*_{j}\, \sigma_{t-j}, \quad \mbox{ for any } t \in \Z,
\end{equation}
where $\theta^*={}^t\big ( c^*_0,\ldots,c^*_p,d^*_1,\ldots,d_q^*\big )\in \Theta $, a compact subset of $\Theta_{p,q}(2)$ defined in \eqref{Thetapq}. Let  $(X_1,\ldots,X_n)$ be an observed trajectory of $(X_t)$. 
Then $\widehat \theta_n={}^t\big ( \widehat c^{(n)}_0,\ldots,\widehat c^{(n)}_p,\widehat d^{(n)}_1,\ldots,\widehat d^{(n)}_q\big ) \limiteasn \theta^*$ where $\widehat \theta_n$ is defined in \eqref{thetaLAV}. 
\end{cor}
The consistency of $\widehat \theta_n$ for long memory LARCH$(\infty)$ can also be derived from Proposition \ref{consLARCH}:
\begin{cor}\label{consLRD}
Assume {\bf $A(2)$} where the sequence $(a_k(\theta))_{k\in \N}$ satisfies \eqref{LRD}, let $\theta^*\in \Theta$ where $\Theta$ is a compact subset of $\Theta(2)$ defined in \eqref{Theta}. Let $(X_1,\ldots,X_n)$ be an observed trajectory of $(X_t)$, which is a stationary long memory LARCH$(\infty)$ solution of \eqref{larch}. Then, under Assumptions $C_\ell(\Theta )$, Id$(\Theta )$ and (S), $\widehat \theta_n \limiteasn \theta^*$ where $\widehat \theta_n$ is defined in \eqref{thetaLAV}.
\end{cor}
\begin{cor}[{\bf Example of long memory LARCH$(\infty)$ studied in Beran and Sch\"utzner, 2009}] \label{corBeran} 
Suppose that $(\xi_t)_{t\in \Z}$ is a sequence of symmetric centered independent random variables, such as $\| \xi_0\|_1=1$ and $\| \xi_0 \|_2<\infty $. Consider $\theta={}^t(a_0,c,d)$ and the sequence $(a_j(\theta))_{j\in \N}$ with $a_0(\theta)=a_0$ and $a_j(\theta)=c \, j^{d-1}$ for $j \geq 1$. Define
\begin{equation}\label{ThetaLRD2}
\Theta=\Big \{ {}^t(a_0,c,d) \in [a_m,a_M]\times [-c_M,c_M]\times [0,d_M],~~c_M^2 \, \|\xi_0\|_2^2 \sum_{k=1}^\infty k^{2\,d_M -2} < 1 \Big \},
\end{equation}
where $0<a_m \leq a_M <\infty$, $0\leq d_M<1/2$ and $0\leq c_M<\infty$. Let $(X_1,\ldots,X_n)$ be an observed trajectory of $(X_t)_{t\in \Z}$ which is a stationary long-memory LARCH$(\infty)$ solution of \eqref{larch} with parameter $\theta^*\in \Theta$. 
Then under Assumption (S),  $\widehat \theta_n \limiteasn \theta^*$ where $\widehat \theta_n$ is defined in \eqref{thetaLAV}.
\end{cor}

\begin{remark}
The condition $\|\xi_0\|_2<\infty$ required in Proposition \ref{consLARCH} and Corollaries \ref{consLARCHp}, \ref{consGLARCH}, \ref{consLRD}, and \ref{corBeran} can be compared with the conditions required in other works dealing with parametric estimation of LARCH processes. In Theorem 4.2. of Francq and Zako\"ian (2010), $\theta^*$ is estimated using a weighted LS estimator (see its definition in \eqref{WLSE}) and the consistency of this LS estimator is established under the condition $\|\xi_0\|_4<\infty$, except for appropriate weights for which $\|\xi_0\|_1<\infty$ is sufficient. In Theorem 4 of Beran and Sch\"utzner (2009), $\theta^*$ is estimated using a smoothed QML estimator (see its definition in \eqref{sQMLE}) and the condition $\|\xi_0\|_3<\infty$ is required for $\L^1$ consistency of this QML estimator. Finally, in Truquet (2014), the strong consistency of a smoothed QML estimator for LARCH$(p)$ processes is obtained under the condition $\|\xi_0\|_s < \infty$ with $s > 0$.
\end{remark}

\subsection{Asymptotic normality of the estimator} 
Under Assumption {\bf $A(4)$}, if $\theta^*\in \interior{\Theta}$, the interior of $\Theta$ where $\Theta$ is a compact subset of $\Theta(4)$, and under Assumption (S) and $C_2(\Theta)$, define, if they exist, the following matrices:
\begin{eqnarray} 
\nonumber \Gamma_1^*&:=&\E \Big [\partial_\theta M_{\theta^*}(0) \, {}^t \partial_\theta M_{\theta^*}(0)\Big ] 
 =\partial_\theta a_0(\theta^*) \, {}^t \partial_\theta a_0(\theta^*)+\sigma_X^2\, \sum_{k=1}^\infty \partial_\theta a_k(\theta^*) \, {}^t \partial_\theta a_k(\theta^*) \\
\label{Gamma1} && \hspace{4.3cm} ~~ \mbox{with} ~\sigma^2_X:= \E \big [X_0^2 \big ]=\frac {a^2_0(\theta^*)\, \sigma^2_\xi \,}{1 -\sigma^2_\xi \sum _{k=1}^\infty a^2_k(\theta^*)}\\ 
\mbox{and}\quad\Gamma_2^*&:=&\E \Big [\big (M_{\theta^*}(0) \big )^2 
\label{Gamma2} \, \partial_\theta M_{\theta^*}(0) \times {}^t \partial_\theta M_{\theta^*}(0)\Big ].
\end{eqnarray} 
Then the asymptotic normality of $\widehat \theta_n$ can be established:
\begin{theo}\label{ANLARCH}
Under Assumption {\bf $A(4)$}, if $\theta^*\in \interior{\Theta}$, is an unknown parameter, where $\Theta$ is a compact subset of $\Theta(4) \subset \R^\ell$, which is defined in \eqref{Theta4}, consider $(X_t)_{t \in \Z}$ as a stationary LARCH$(\infty)$ solution of \eqref{larch} and $(X_1,\ldots,X_n)$ an observed trajectory of $(X_t)$. Assume that Assumption (S) is satisfied as well as Assumption $C_{\ell+2}(\Theta)$ and Id$(\Theta)$. Then, if the matrices $\Gamma_1^*$ and $\Gamma_2^*$ defined in \eqref{Gamma1} and \eqref{Gamma2} are positive definite,
\begin{equation}\label{CLT}
\sqrt n \, \big (\widehat \theta_n -\theta^* \big ) \limiteloin {\cal N} \Big ( 0\, , \,(\sigma_\xi^2 -1) \, \big (\Gamma_1^* \big )^{-1}\Gamma_2^* \, \big (\Gamma_1^* \big )^{-1}\Big) 
\end{equation}
\end{theo}
\begin{remark}\label{Slu}
The expression of $\Gamma_2^*$ is not easy to simplify even in the simplest cases. This is not really a problem since, as usual, one can use the Slutsky Lemma, after defining the following estimators of $\sigma_\xi^2$, $\Gamma_1^*$ and $\Gamma_2^*$ by 
\begin{eqnarray*}
&& \widehat \sigma_\xi^2:= \frac 1 n \, \sum_{t=1}^n \frac { X_t^2}{\Big (a_0(\widehat \theta_n)+\sum_{k=1}^{t-1} a_k(\widehat \theta_n)\, X_{t-k} \Big )^2};  \\
&&\displaystyle \widehat \Gamma_1:=\frac 1 n \, \sum_{t=1}^n \Big (\partial_\theta a_0(\widehat \theta_n)+\sum_{k=1}^{t-1} \partial_\theta a_k(\widehat \theta_n)\, X_{t-k} \Big )\, {}^t \Big (\partial_\theta a_0(\widehat \theta_n)+\sum_{k=1}^{t-1} \partial_\theta a_k(\widehat \theta_n)\, X_{t-k} \Big );\\
&&\displaystyle \widehat \Gamma_2:=\frac 1 n \, \sum_{t=1}^n \Big (a_0(\widehat \theta_n)+\sum_{k=1}^{t-1} a_k(\widehat \theta_n)\, X_{t-k} \Big )^2 \Big (\partial_\theta a_0(\widehat \theta_n)+\sum_{k=1}^{t-1} \partial_\theta a_k(\widehat \theta_n)\, X_{t-k} \Big ) \times \\
&&\hspace{7cm}\times  {}^t \Big (\partial_\theta a_0(\widehat \theta_n)+\sum_{k=1}^{t-1} \partial_\theta a_k(\widehat \theta_n)\, X_{t-k} \Big ).
\end{eqnarray*}
Note that the consistency of $\widehat \sigma_\xi^2$ was proved in Francq and Zako\"ian (2010). The matrix $\widehat \Gamma_1$ and $\widehat \Gamma_2$ are also consistent estimators of $\Gamma_1^*$ and $\Gamma_2^*$ (see the proof in section \ref{Proofs}), and therefore
\begin{equation}\label{CLT2}
\sqrt n \,\big (\widehat \sigma_\xi^2 -1 \big )^{-1/2} \big (\widehat \Gamma_1 \big )^{1/2} \big ( \widehat \Gamma_2 \big )^{-1/2} \, \big (\widehat \Gamma_1\big )^{1/2} \big (\widehat \theta_n -\theta^* \big ) \limiteloin {\cal N} \big ( 0\, , \,I_d\big).
\end{equation}
Such a central limit theorem \eqref{CLT2}, satisfied by $\widehat \theta_n$ allows the computation of asymptotic confidence intervals or test thresholds on $\theta$. 
\end{remark}
The special case of GLARCH$(p,q)$ processes can be considered under simplified assumptions:
\begin{cor}\label{corGLARCH}
Assume the conditions of Corollary \ref{consGLARCH}, and also assume $\|\xi_0\|_4<\infty$ Let $\Theta$ be a compact set included in  $\Theta_{p,q}(4)$ and $\theta^*={}^t\big ( c^*_0,\ldots,c^*_p,d^*_1,\ldots,d_q^*\big ) \in \interior{\Theta}$. Then, with $\widehat \theta_n={}^t\big ( \widehat c^{(n)}_0,\ldots,\widehat c^{(n)}_p,\widehat d^{(n)}_1,\ldots,\widehat d^{(n)}_q\big )$, the central limit theorems \eqref{CLT} and \eqref{CLT2} hold. 
\end{cor}
In the case of LARCH$(p)$ processes, which are particular cases of GLARCH$(p,q)$ processes, we can go further into the details of the conditions for asymptotic normality:
\begin{cor}\label{corLARCH}
Assume the conditions of Corollary \ref{consLARCH}, and suppose also $\|\xi_0\|_4<\infty$, $\Theta$ defined by 
\begin{equation}\label{Thep}
\Theta=\Big \{{}^t(a_0,a_1,\ldots,a_p)\in [\underline a,\overline a]\times \R^p~~\mbox{with}~~\|\xi_0 \|_4^4 \sum_{j=1}^{p} a^{4}_j+ 6 \,
\| \xi_0\|^2_2 \, \sum_{j=1}^{p} a^{2}_j \leq r \Big \},
\end{equation}
where $0<\underline a<\overline a$ and $0<r<1$, and $\theta^*={}^t(a^*_0,a^*_1,\ldots,a^*_p)\in \interior{\Theta}$. 
Then, with $\widehat \theta_n={}^t\big ( \widehat a^{(n)}_0,\widehat a^{(n)}_1,\ldots,\widehat a^{(n)}_p\big )$, the central limit theorems \eqref{CLT} and \eqref{CLT2} hold. 
%Moreover, for LARCH$(p)$ processes, considering the usual representation \eqref{M2} described in Remark \ref{Rem1}, we obtain:
%\begin{equation}\label{a'}
%\sqrt n \,\big (\widehat \sigma_\xi^4 -\widehat \sigma_\xi^2 \big )^{-1/2}   \big (\widehat \Gamma_1  \, \widehat \Gamma_2  ^{-1} \widehat \Gamma_1\big )^{1/2}  \Big ({}^t \big (\widehat a'_0,\ldots,\widehat a'_p \big ) -{}^t \big ( a_0^{'*},\ldots, a^{'*}_p \big ) \Big ) \limiteloin {\cal N} \big ( 0\, , \,I_{p+1}\big).
%\end{equation}
\end{cor}
As an example of computation of the asymptotic covariance, if we consider the case of a LARCH$(1)$ process, we obtain:
$$
\Gamma_1^*= \left (\begin{array}{cc} 1 & 0 \\ 0 & \sigma^2_X  \end{array} \right ) \quad \mbox{and}\quad \Gamma_2^*= \left (\begin{array}{cc} a_0^2+\sigma^2_X  & 2 \, a_0a_1 \sigma^2_X  \\ 2 \, a_0a_1 \sigma^2_X & a_0^2 \sigma^2_X+\E[X_0^4]  \end{array} \right ).
$$
This implies 
$$
\big (\Gamma_1^* \big )^{-1}\Gamma_2^* \, \big (\Gamma_1^* \big )^{-1}=\left (\begin{array}{cc} a_0^{*2}+\sigma^2_X  & 2 \, a^*_0a^*_1   \\ 2 \, a^*_0a^*_1 & \frac {a_0^{*2}}{\sigma^2_X}+\frac{\E[X_0^4]}{\sigma^4_X}  \end{array} \right ),
$$
where $\left \{ \begin{array}{ccl}
\sigma^2_X&=&a_0^{*2} \, \sigma^2_\xi \,\big (1-\sigma^2 _\xi a_1^{*2} \big )^{-1} \\
\E[X_0^4]&=&a_0^{*4} \, \E [\xi_0^4] \, \big ( 1+5 \,\sigma^2 _\xi a_1^{*2} \big )\big (1-\sigma^2 _\xi a_1^{*2}\big )^{-1}\big (1-\E [\xi_0^4] a_1^{*4} \big )^{-1}
\end{array} \right .
.$ \\
~\\
The asymptotic normality of the estimator $\widehat \theta_n$ can also be obtained in the case of long-memory LARCH$(\infty)$:
\begin{cor}\label{corLRD}
Assume the conditions of Corollary \ref{consLRD}, and further assume that $\|\xi_0\|_4<\infty$, $\Theta$ is a compact set included in $\Theta(4)\subset \R^\ell$ and $\theta^* \in \interior{\Theta}$.
Then, under Assumption $C_{\max(\ell,2)}(\Theta)$. and if $\Gamma_1^*$ and $\Gamma_2^*$ are positive definite matrices, the central limit theorems \eqref{CLT} and \eqref{CLT2} hold. 
\end{cor}
\begin{cor}[{\bf Example of long memory LARCH$(\infty)$ studied in Beran and Sch\"utzner, 2009}] \label{corBeran2} Under the assumptions of Corollary \ref{corBeran} and if $~\overline c^4  \|\xi_0\|_4^4 \sum_{i=1}^\infty j^{4\overline d -4} \!+\! 6  \overline c^2  \|\xi_0\|_2^2 \sum_{i=1}^\infty j^{2\overline d -2}\! < \!1$, then the central limit theorems \eqref{CLT} and \eqref{CLT2} hold. 
\end{cor}
\begin{remark}
To our knowledge, the only result for estimating the memory parameter $d$ in the case of LARCH$(\infty)$ processes with long memory was obtained in Beran and Sch\"utzner (2009) using the smoothed QML estimator. However, the expression of this estimator uses only a small part of the sample $(X_1,\ldots,X_n)$, namely the last $n^{\beta}$ observations, where $\beta < 1-2d$, while $d$ is unknown to account for the long memory of the process. This leads to a convergence rate of $n^{\beta/2}$ for this truncated QML estimator, which is far less interesting than the convergence rate of $\sqrt n$ obtained with $\widehat \theta_n$. Monte Carlo experiments (see section \ref{Simu}) will also show that $\widehat \theta_n$ performs numerically better than these other estimators in terms of convergence rate, especially when compared to estimators based on QML. 
\end{remark}
\section{Numerical experiments} \label{Simu}
\subsection{Monte Carlo experiments}
In this section, we report the results of Monte Carlo experiments conducted with different LARCH processes. More specifically, we considered:
\begin{itemize}
\item Three different LARCH processes:
\begin{enumerate}
\item A LARCH$(2)$ process, with parameters $a_0=5$, $a_1=-0.2$ and $a_2=0.4$;
\item A GLARCH$(1,1)$ process, with parameters $c_0=2$, $c_1=0.3$ and $d_1=-0.6$;
\item A long memory LARCH$(\infty)$ process, with $\theta={}^t(a_0,c,d)$ and $a_0(\theta)=a_0$ and $a_k(\theta)=c \, k^{d-1}$. We choose $a_0=1$, $c=0.2$ and $d=0.1,~0.2,~0.3$ and $0.4$, using the same example studied in Beran and Sch\"utzner (2009) for its numerical illustrations. To define numerically the trajectory of such processes, whose theoretical definition involves an infinite sum, a truncation of this infinite sum has been used by taking a sum of one million terms in the computation of $M_\theta(t)$, which also requires the generation of one million additional realizations of the white noise.   
\end{enumerate}
\item Several trajectory lengths: $n=200,~500,~1000, ~2000$ and $5000$ for LARCH$(2)$ and GLARCH$(1,1)$ processes, and $n=1000,~2500,~5000$ and $10000$ for the LARCH$(\infty)$ process (as in Beran and Sch\"utzner, 2009);
\item Two distributions for $\xi_0$ such as $\E[|\xi_0|]=1$: a Gaussian ${\cal N}(0,\pi/2)$ distribution denoted ${\cal N}$ and a normalized Student $t(6)$ distribution with $6$ freedom degrees.
\item Choice of $\Theta$ for minimisation \eqref{thetaLAV} and calculation of $\widehat \theta_n$: as already mentioned in the Remark \ref{Rem00}, the set $\Theta(2)$ depends on the distribution of the noise, and an "extended" condition will be used in the minimisation algorithm. Hence, for  the LARCH$(2)$ process this means $0\leq a_1,\,a_2\leq 1$, for the GLARCH$(1,1)$ process it will be $0\leq c_1,\,d_1\leq 1$ and for the long memory LARCH$(\infty)$  process, $c<1$ and $0\leq d \leq 0.5$. After computing the estimator, and using the empirical variance of the residuals $\widehat \xi_t$ as an estimator of $\|\xi_0\|_2^2$, the condition $\widehat \theta_n \in \Theta(2)$ can be checked from the equation $\widehat \sigma^2_{\widehat \xi_t}\, \sum_{j=1}^\infty a^2_j(\widehat \theta_n)<1$, which means for example that $\displaystyle \widehat \sigma^2_{\widehat \xi_t}\,\widehat c^2 \, \zeta(2-2\, \widehat d)<1 $ for the long memory LARCH$(\infty)$  process, where $\zeta$ is the Riemann zeta function.  %is obtained from the emp\big (\zeta(2-2\,d)\big )^{-1/2}$ %the one for which $\|\xi_0\|_2=1$, {\it i.e.} to require $\theta \in \R^\ell$ such that $\displaystyle 
%\sum_{j=1}^{\infty} a^2_j(\theta) < 1$. Hence, for  the LARCH$(2)$ process this means $a_1^2+a_2^2<1$, for the GLARCH$(1,1)$ process it will be $c_1^2+d_1^2<1$ and for the long memory LARCH$(\infty)$  process, $c<\big (\zeta(2-2d)\big )^{-1/2}$ where $d<1/2$ and $\zeta$ is the Riemann zeta function. After \\
\end{itemize}
For each choice of process, length $n$ and noise distribution, $1000$ replications of independent trajectories of the LARCH process are generated, except for the long memory LARCH$(\infty)$ where only $300$ replications are used due to the computational time of the estimator. Note that in this case, an algorithm based on the Fast Fourier Transform (FFT) has been developed in Nielsen and No\"el (2021, section 2.2), which could have been used to speed up the computations significantly, or if we had applied the estimators to series larger than $10^4$.  \\
~\\
Two other estimators are to be compared with $\widehat \theta_n$:
\begin{enumerate}
\item Following Beran and Sch\"utzner (2009) and Truquet (2014), the first is the smooth approximation of the QMLE, which for $h > 0$ is given by
\begin{equation}\label{sQMLE}
\widehat \theta_{QML}(h):=\mbox{Arg}\! \min_{\! \! \! \!\! \! \theta \in \Theta} \, \frac1 n \, \sum_{t=1}^n \frac {h+X_t^2}{h+(M_\theta(t))^2} +\log \big (h+(M_\theta(t))^2\big ).
\end{equation} 
The a priori choice of $h$ or data-driven $\widehat h$ is not a straightforward task, although Truquet has provided some guidance in Truquet (2014). Therefore, we will present the results obtained for $2$ different values of $h$ that give the best performances. In the case of the considered LARCH$(\infty)$ process with long memory, Beran and Sch\"utzner (2009) proposed a modified version of $\widehat \theta_{QML}(h)$ and we will use their results.
\item Following Francq and Zako\"ian (2010), the second is the weighted least squares estimator defined by:
\begin{equation}\label{WLSE}
\widehat \theta_{FZ}:=\mbox{Arg}\! \min_{\! \! \! \!\! \! \theta \in \Theta} \,  \frac1 n \,  \sum_{t=1}^n \tau_t \, \big (X_t^2-(M_\theta(t))^2 \big )^2,
\end{equation} 
where the weights $(\tau_t)$ are obtained for LARCH$(p)$ or GLARCH$(p,q)$ using an empirical rule proposed in Ling (2007): 
$\displaystyle \tau_t=\Big (\max \big ( 1 \, , \, \frac 1 C \, \sum_{i=1}^p |X_{t-i}| \, \1_{|X_{t-i}|>C}  \big )\Big )^{-4},
$
where $C$ is computed as the 90\% quantile of the absolute values $\big (|X_1|,\ldots,|X_n| \big )$. In case of long memory LARCH$(\infty)$, we replace  $p$ by $t-1$ in the definition of $\tau_t$.  
\end{enumerate}
\begin{remark}
The estimator $\widehat \theta_n$ can also be defined as a nonlinear least squares estimator for the regression $(|X_t|)_{1\leq t\leq n}$ on $\Big ( \Big |
a_0(\theta) + \sum_{j=1}^\infty a_j(\theta)\, X_{t-j} \Big |\Big )_{1\leq t\leq n}$.  
As it was already done by Francq and Zako\"ian (2010) (see the estimator defined in \eqref{WLSE}), a weighted version of $\widehat \theta_n$ could also be considered:
$$
\widehat \theta^w_n=\mbox{Arg}\min_{\! \! \! \!\! \! \theta \in \Theta} \, \frac1 n \, \sum_{t=1}^n w_t \, \Big ( |X_t|-\Big |
a_0(\theta) + \sum_{j=1}^{\infty}  a_j(\theta)\, \widetilde X_{t-j} \Big | \Big )^2,
$$
where $w_t=W\big ((X_{t-k})_{1\leq k\leq t-1}\big )>0$, defining a sequence of selected weights. Inspired by their work, it seems that using weights of the form $w_t=\Big (\max \big ( 1 \, , \, \frac 1 C \, \sum_{i=1}^p |X_{t-i}| \, \1_{|X_{t-i}|> C} \big )^{-2}$ would allow, in particular, remove the condition $r= 4$ from the assumptions of asymptotic normality of $\widehat \theta_n$ by requiring only $r=1$ for $\widehat \theta^w_n$. To improve the convergence rate of the $\widehat \theta_n$, taking into account Remark 4.2 of Francq and Zako\"ian (2010), one might think that $w_t=\big (M_{\theta^*}(t)\big )^{-2}$ would be the ideal sequence of weights, but infeasible in practice. 
\end{remark}
\begin{remark}
The estimator $\widehat \theta_n$ is obtained with parameters defined under the normalization condition $\|\xi_0\|_1=1$, while the estimators $\widehat \theta_{QML}(h)$ or $\widehat \theta_{ FZ }$ are basically defined under the normalization condition $\|\xi_0\|_2=1$. After remark \ref{Rem1}, a suitable renormalization of the parameters of the LARCH process allows the transition from one condition to the other. So, for reference, consider the LARCH$(p)$ equation with parameter $\theta={}^t(a_0,\ldots,a_p)$ under condition $\|\xi_0\|_1=1$, $\widehat \theta_n$ is an estimator of $\theta$, but $\widehat \theta_{QML}(h)$ or $\widehat \theta_{ FZ }$ are estimators of $\theta /\|\xi_0\|_2$. Therefore, a comparison of the accuracies of the estimators is possible by considering $\widehat \theta_n$, $\|\xi_0\|_2 \,\widehat \theta_{QML}(h)$ and $\|\xi_0\|_2 \, \widehat \theta_{ FZ }$. If the law of $\xi_0$ is unknown, the comparison with the estimator $\widehat \sigma_\xi$ defined in \eqref{CLT2} is still possible.
\end{remark} 
The results are presented in tables \ref{Table1}, \ref{Table2} and \ref{Table3}. \\
~\\
\begin{table*}
\centering
\begin{tabular}{l|l||c|c|c||c|c|c||c|c|c||c|c|c|}
& &  \multicolumn{3}{c|}{$\widehat \theta_n$} & \multicolumn{3}{c|}{$\widehat \theta_{FZ}$} & \multicolumn{3}{c|}{$\widehat \theta_{QML}(2)$} &\multicolumn{3}{c|}{$\widehat \theta_{QML}(1)$}  \\
\hline 
$\xi_0$ law & $n$ & $a_0$ &$a_1$ &$a_2$ & $a_0$ &$a_1$ &$a_2$& $a_0$ &$a_1$ &$a_2$& $a_0$ &$a_1$ &$a_2$  \\ 
\hline
${\cal N}$ & $200$ & 0.326 &0.047 &0.064 &0.423 &0.092 &0.100 &1.015 &0.136 &0.123 &1.832 &0.248 &0.168\\
&$500$ & 0.210 &0.029 &0.043 &0.268 &0.059 &0.065 &0.446 &0.070 &0.086 &1.119 &0.145 &0.121\\
&$1000$ &0.145 &0.021 &0.030 &0.188 &0.044 &0.047 &0.382 &0.060 &0.082 &0.582 &0.075 &0.084\\
&$2000$ & 0.101 &0.014 &0.021 &0.130 &0.030 &0.033 &0.265 &0.047 &0.059 &0.453 &0.053 &0.080\\
&$5000$ & 0.065 & 0.009 &0.013 &0.083 &0.019 &0.021 &0.205 &0.031 &0.048 &0.326 &0.037 &0.061\\
\hline
$t(6)$ & $200$ & 0.433 &0.061 &0.091& 0.624 &0.130& 0.133&2.111 &0.268 &0.205 &1.916& 0.258 &0.186\\
&$500$ & 0.272 &0.040 &0.061 &0.390& 0.091& 0.094&1.699 & 0.214& 0.186& 1.569& 0.209&0.167\\
&$1000$ & 0.224 &0.029 &0.051 &0.275& 0.067& 0.071 &1.361&0.170& 0.191& 1.270& 0.166& 0.155\\
&$2000$ & 0.124 &0.021 &0.031 &0.196& 0.048& 0.051  &1.334& 0.149& 0.175& 1.535& 0.147 &0.151\\
&$5000$ & 0.077 & 0.014 &0.021 &0.127& 0.031& 0.033 &1.230& 0.136& 0.211& 1.533& 0.147 &0.188\\
\hline
\end{tabular}
\caption{Square roots of the MSE computed for each  estimator of parameters $a_0=5$, $a_1=-0.2$ and $a_2=0.4$ of a LARCH$(2)$ process computed from $1000$ independent replications. }
\label{Table1}
\end{table*}
\begin{table*}
\centering
\begin{tabular}{l|l||c|c|c||c|c|c||c|c|c||c|c|c|}
& &  \multicolumn{3}{c|}{$\widehat \theta_n$} & \multicolumn{3}{c|}{$\widehat \theta_{FZ}$} & \multicolumn{3}{c|}{$\widehat \theta_{QML}(1)$} &\multicolumn{3}{c|}{$\widehat \theta_{QML}(0.5)$}  \\
\hline 
$\xi_0$ law & $n$ & $c_0$ &$c_1$ &$b_1$ & $c_0$ &$c_1$ &$b_1$& $c_0$ &$c_1$ &$b_1$& $c_0$ &$c_1$ &$b_1$ \\ 
\hline
${\cal N}$ & $200$ & 0.172 &0.044 &0.096 &0.238 &0.094 &0.135 &0.179 &0.045 &0.099 &0.190 &0.052 &0.105 \\
&$500$ &0.108 &0.028 &0.057 &0.158 &0.068 &0.081 &0.102 &0.031 &0.055 &0.114 &0.040 &0.061 \\
&$1000$ & 0.071& 0.019 &0.039& 0.113 &0.050& 0.055& 0.066& 0.018& 0.034 &0.081& 0.029 &0.048 \\
&$2000$ & 0.052 & 0.013& 0.028& 0.087 &0.040 &0.039& 0.045 &0.012& 0.023& 0.050& 0.019& 0.024 \\
&$5000$ & 0.033 &0.008 &0.017 &0.065 &0.030 &0.025 &0.028 &0.007 &0.014 &0.027 &0.006 &0.012\\
\hline
$t(6)$ & $200$ & 0.233& 0.061 &0.145 &0.367 &0.113& 0.232& 0.335 &0.085 &0.191& 0.349& 0.094 &0.181  \\
&$500$ & 0.142 &0.042 &0.081 &0.196 &0.077 &0.116 &0.255 &0.072 &0.156& 0.275 &0.095& 0.156\\
&$1000$ & 0.091& 0.029 &0.051 &0.147 & 0.063 &0.082& 0.207 &0.064& 0.125& 0.261& 0.096& 0.144\\
&$2000$ & 0.064  & 0.020& 0.033& 0.096& 0.043& 0.052& 0.153 &0.059& 0.097& 0.214& 0.092& 0.119 \\
&$5000$ & 0.039& 0.013 &0.022 &0.071& 0.029 &0.037& 0.116& 0.036 &0.095& 0.184& 0.073& 0.119\\
\hline
\end{tabular}
\caption{Square roots of the MSE computed for each  estimator of parameters $c_0=2$, $c_1=0.3$ and $d_1=-0.6$ of a GLARCH$(1,1)$ process computed from $1000$ independent replications. }
\label{Table2}
\end{table*}
\begin{center}
\begin{table*}
\begin{center}
\begin{tabular}{|l|l||c|c|c||c|c|c||c|}
& &  \multicolumn{3}{c||}{$\widehat \theta_n$} & \multicolumn{3}{c||}{$\widehat \theta_{FZ}$} & \multicolumn{1}{c||}{$\widehat d_{QML}$} \\
\hline
$d$& $n$ & $a_0$ & $c $ & $d$ & $a_0$ & $c $ & $d$ & $d$ \\ \hline
$d=0.1$  &$1000$& 0.035& 0.024& 0.089 &0.092& 0.054 &0.160 & 0.357\\
&$2500$&0.020 &0.015 &0.048 &0.089 &0.059 &0.119& 0.292\\
&$5000$& 0.016 &0.010 &0.036& 0.040& 0.031 &0.084&0.217 \\
&$10000$& 0.013 &0.010 &0.021& 0.033 &0.029 &0.053& 0.198\\
$d=0.2$ &$1000$&0.041 &0.023& 0.060& 0.103& 0.059& 0.147 &1.449 \\
&$2500 $& 0.028& 0.017& 0.043 &0.052& 0.033 &0.088&0.733 \\
&$5000$& 0.016 &0.010 &0.024 &0.033 &0.027& 0.050& 0.559\\
&$10000$& 0.014 &0.008& 0.017& 0.032 &0.024& 0.045&  0.257\\
$d=0.3$ &$1000$&0.060& 0.024& 0.053 &0.119 &0.054 &0.110 &- \\
&$2500 $& 0.045 &0.017& 0.037& 0.091 &0.035 &0.087&-\\
&$5000$& 0.026 &0.011& 0.022& 0.057 &0.029& 0.050&-\\
&$10000$& 0.016 &0.008& 0.017& 0.039 &0.024 &0.042& -\\
$d=0.4$ &$1000$& 0.100 &0.025& 0.047 &0.189 &0.051 &0.081&- \\
&$2500 $& 0.071 &0.017 &0.030 &0.158 &0.041 &0.074&-\\
&$5000$& 0.041 &0.011& 0.018 &0.122& 0.032 &0.057&-\\
&$10000$& 0.034 &0.009& 0.015& 0.067& 0.024& 0.036& -\\
\hline
\end{tabular}
\end{center}
\caption{Square roots of the MSE computed for estimators $\widehat \theta_n$ and $\widehat \theta_{FZ}$ of parameters $a_0=1$, $c=0.2$ and $d=0.1,\,0.2,\,0.3$ and $0.4$ of the LARCH$(\infty)$ process computed from $300$ independent replications, and for $\widehat d_{QML}(h)$ already computed in Beran and Sch\"utzner (2009) for $d=0.1$ and $0.2$. }
\label{Table3}
\end{table*}
\end{center}
~\\
\noindent{\bf Conclusions of the Monte Carlo experiments:} 
\begin{itemize}
\item Looking at the decay of the square root of the MSE of the $\widehat \theta_n$ components towards 0 as $n$ increases from 200 to 5000 (or from 1000 to 10000), we see that this decay is approximately $1/\sqrt n$, which corresponds to the theoretical convergence rate established in Theorem \ref{ANLARCH}. This is true for a LARCH$(2)$ process as well as for a GLARCH$(1,1)$ or a LARCH$(\infty)$ process with long memory, regardless of whether we consider a white noise with a Gaussian or Student distribution $t(6)$.
\item In general, the square root of the MSE of $\widehat \theta_n$ converges to $0$ twice as fast as that of $\widehat \theta_{ FZ }$ when the white noise follows a Gaussian or Student $t(6)$ distribution, for the three processes considered. We note that $\widehat \theta_{ FZ }$ gives convincing results in the case of a LARCH$(\infty)$ process with long memory, while this was not shown in Francq and Zako\"ian (2010), although they are far outperformed by those of $\widehat \theta_n$.
However, the numerical results for the convergence of the MSE of $\widehat \theta_{ FZ }$ towards $0$ become much worse when the white noise distribution of the LARCH processes is a Student $t(6)$ distribution.

\item Finally, $\widehat \theta_{QML}$ gives satisfactory performances comparable to those of $\widehat \theta_n$ only in one case, namely for the LARCH$(1,1)$ process with Gaussian white noise, and this after choosing an optimal regularisation parameter. It should be noted, however, that the choice of this parameter can easily be made in the context of Monte Carlo experiments, but this would otherwise require a data-driven procedure that does not currently exist. For the estimation of the long memory parameter $d$ in the case of a LARCH$(\infty)$ process, the QML estimator proposed in Beran and Sch\"utzner (2009) has truly disastrous performances compared to those obtained with $\widehat \theta_{ FZ }$ and in particular those of $\widehat \theta_n$.
\end{itemize}
\subsection{Application on the returns of a financial index}
Let us look at the financial returns of an index called the CAC40, which is the French equivalent of the FTSE100, with 40 companies instead of 100. These are the closing values between 15 November 2002 and 15 November 2022, i.e. 20 years and $n=5122$ data. The correlogram plot shows a weak white-noise type behaviour, which is also consistent with a conditionally heteroscedastic process. 
\begin{figure}[h]
  \hspace*{0cm}
	\centering
  \includegraphics[height=6cm,width=17cm]{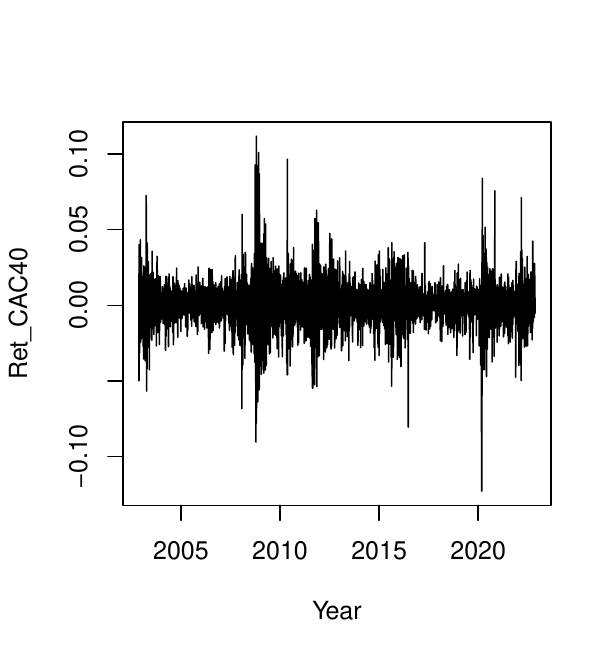}
	\caption{Financial returns of CAC40 index, between November 15, 2002 and November 15, 2022\label{Fig1}} 
\end{figure}
Firstly, we can see that these returns have values extremely close to 0 (70 are in absolute terms less than $1 \%$ of the standard deviation) and even twice exactly 0, which would be very unlikely with a GARCH process often used to model this data, for which the conditional variance is always greater than a positive constant. \\
Moreover, as already pointed out in Giraitis {\it i.e.} (2004) with the S\&P500 returns data since 1928, we can observe a particular behaviour of the leverage estimate: it is negative for almost all lags and follows a power-law type distribution (see Figure \ref{Fig2}). A non-linear least squares approximation gives the value of this power $\simeq -0.55$. If we note $h_t$ the leverage, this would mean that $h_t \sim C\, t^{d-1}$ with $d\simeq 0.45$ and $C<0$. 
\begin{figure}[h]
  \hspace*{0cm}
	\centering
  \includegraphics[height=6cm,width=17cm]{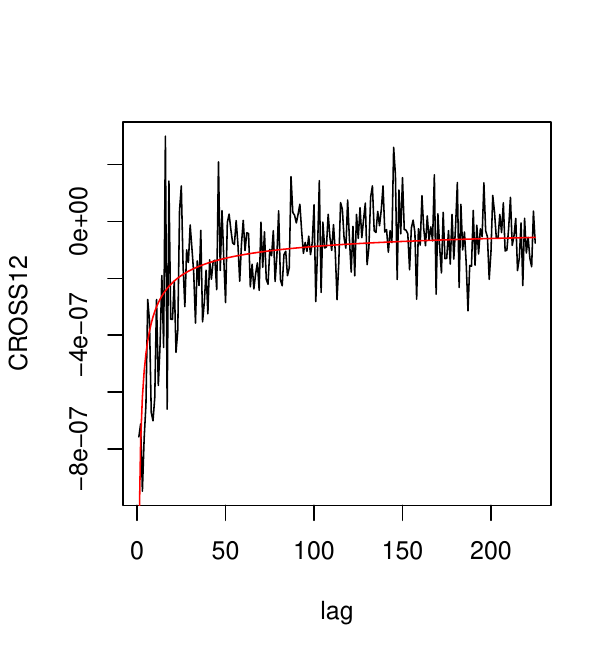}
	\caption{Leverage estimate of the financial returns of CAC40 index, between November 15, 2002 and November 15, 2022\label{Fig2}} 
\end{figure}
Now, if the CAC40 returns follow the example of LARCH$(\infty)$ processes with long memory $(X_t)$ studied in Beran and Sch\"utzner (2009), {\it i.e.} $a_j(\theta)=c\, j^{d-1}$ for $j\geq 1$, then we have:
$$
\cov(X_0,X_t)=0, \quad h_t=\cov(X_0,X_t^2)\simeq C\, t^{d-1} \quad \mbox{and}\quad \cov(X_0^2,X^2_t)\simeq c'\, t^{2d-1},
$$
with $0<d<1/2$, $C<0$ and $c'>0$ (see again Giraitis {\it i.e.}, 2004, or Robinson and Zaffaroni, 1997). We check that this last asymptotic behaviour is well verified by plotting the correlogram of the squares of the process, which is done in figure \ref{Fig3}. Again we observe a power law type behaviour, and a non-linear least squares estimation of this power gives the result $\simeq -0.09$. However, if we use the value of $d=0.445$ obtained numerically from $h_t$, we find that $2d-1\simeq -0.11$, a value very close to $-0.09$. There seems to be a good fit of a long memory LARCH$(\infty)$ model with these return data. 
\begin{figure}[h]
  \hspace*{0cm}
	\centering
  \includegraphics[height=6cm,width=17cm]{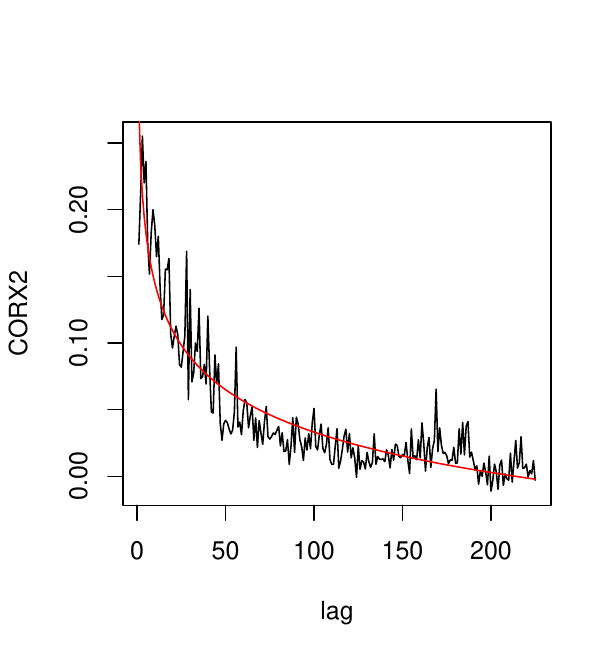}
	\caption{Correlogram of the squared financial returns of CAC40 index, between November 15, 2002 and November 15, 2022\label{Fig3}} 
\end{figure}
~\\
{\bf Fitting the returns of CAC40 index with two models of long memory LARCH$(\infty)$} \\
~\\
{\bf a.} We first fitted these data with  the long memory  LARCH$(\infty)$ process studied in Beran and Sch\"utzner (2009), {\it i.e.}, such that 
$$
a_0(\theta)=a_0\quad\mbox{and}\quad a_j(\theta)=c\, j^{d-1}\quad\mbox{for}~ j\geq 1.
$$
For this model, we considered the  estimator $\widehat \theta_n={}^t\big (\widehat a_0,\widehat c,\widehat d \big )$ used in the results of Table 3 and we obtain 
$$
\widehat a_0\simeq 0.010\quad,~\widehat c \simeq -0.159\quad\mbox{and}\quad \widehat d\simeq 0.488.
$$ 
This is quite consistent with the values of $d$ previously obtained by nonlinear least squares. Note that we also obtained $\widehat \theta_{ FZ }={}^t\big (0.007, -0.143,0.497\big )$, a value quite close to that of $\widehat \theta_n$, confirming the long memory property of this time series. Note also that using \eqref{CLT2} an estimate of the covariance matrix of these estimators can be computed. And we obtained the following $95\%$ confidence intervals from $\widehat \theta_n$: 
$$
a_0 \in \big [0.00999\, , \,0.0110 \big ], \quad c \in \big [-0.192\, , \, -0.125\big ]\quad \mbox{and}\quad d \in \big [0.433\, , \, 0.544\big ].
$$\\
However, we notice that $\sum_{j=1}^\infty a^2_j(\widehat \theta_n)\simeq (0.159)^2 \, \zeta(1.024) \simeq 1. 068>1$ and also $\sum_{j=1}^\infty a^2_j(\widehat \theta_{ FZ })\simeq (-0.143)^2 \, \zeta(1.006)>1$. Thus, whatever the distribution of the noise, the second order stationarity condition is not verified. The long memory model LARCH$(\infty)$ process such as $a_j(\theta)=c\, j^{d-1}$ for $j\geq 1$ does not seem appropriate. This is also confirmed by the correlograms of the residuals $\widehat \xi_t=X_t/(\widehat a_0+\widehat c \, \sum_{j=1}^{t-1} j^{\widehat d -1} X_{t-j})$ and their absolute values $|\widehat \xi_t |$ plotted in Figure \ref{Fig4}. Indeed, and it is particularly clear with the correlogram of $|\widehat \xi_t |$, $(\widehat \xi_t)$ could not appear almost as white noise (note that Francq and Zako\"ian, 2010, also proposed a goodness-of-fit test based on $(\widehat \xi_t^2)$, but only for AR$(q)$-LARCH$(p)$ processes). Modelling these data with this model is therefore problematic. 
\begin{figure}[h]
  \hspace*{0cm}
	\centering
  \includegraphics[height=6cm,width=8cm]{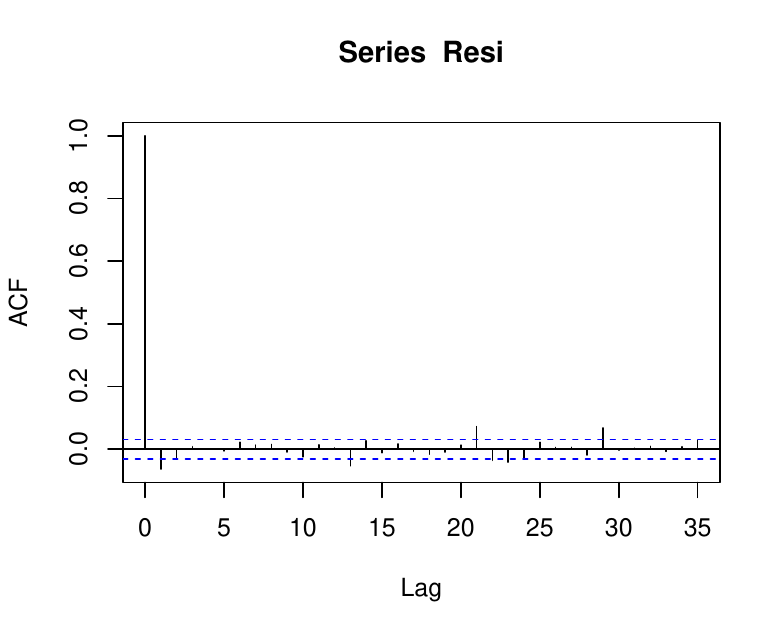} \includegraphics[height=6cm,width=8cm]{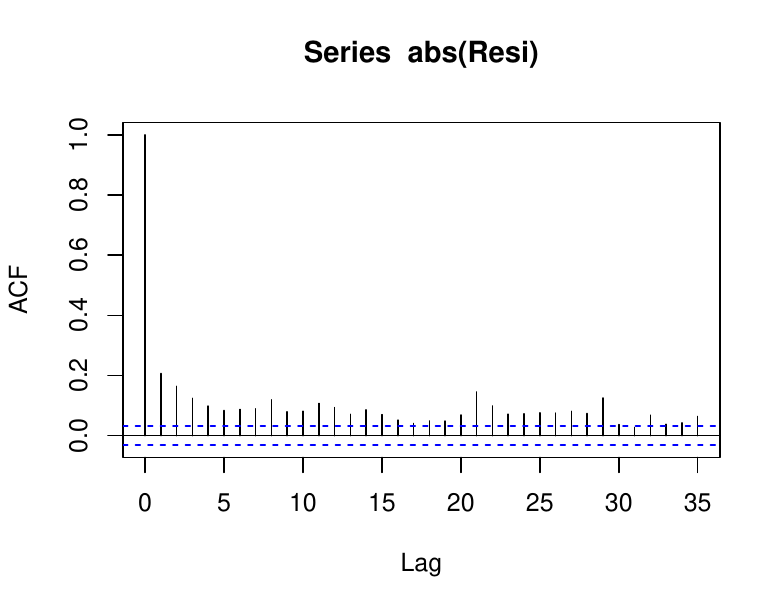}
	\caption{Correlograms of residuals and absolute values of residuals for the returns of CAC40 index, between November 15, 2002 and November 15, 2022, using $a_j(\theta)=c\, j^{d-1}$ for $j\geq 1$\label{Fig4}} 
\end{figure}
~\\
{\bf b.} We now consider another model of a long memory LARCH$(\infty)$ process. We will thus suppose that 
$$
a_0(\theta)=a_0\quad\mbox{and}\quad a_j(\theta)=c\,j^{d-1}\big (1 +\frac {c'}j \big ) \, \quad\mbox{for}~ j\geq 1,
$$
with $\theta=(a_0,c,d,c') \in (0,\infty)\times \R \times (0,1/2)\times \R$. It is clear that for $c'=0$ we recover the previous model. For this model, we computed on data the  estimator $\widehat \theta_n={}^t\big (\widehat a_0,\widehat c,\widehat d,\widehat c' \big )$  and obtained
$$
\widehat a_0\simeq 0.010\quad,~\widehat c \simeq -0.302, ~\widehat d \simeq 0.346\quad\mbox{and}\quad \widehat c'\simeq -0.678.
$$ 
In contrast to the previous model, now $\sum_{j=1}^\infty a^2_j(\widehat \theta_n)\simeq (-0.302)^2 \, \big (\zeta(1.308)+(0.678)^2\, \zeta(3.308)-1.356 \,\zeta(2.308) \big ) \simeq 0.222<1$, and the second order stationary condition is generally satisfied (except when $\|\xi_0\|_2 \geq 4.5$). And using \eqref{CLT2}, an estimate of the covariance matrix of these estimators can be computed, and the following $95\%$ confidence intervals of the parameters are obtained: 
$$
a_0 \in \big [0.0099\, , \,0.0110 \big ], \quad c \in \big [-0.383\, , \, -0.221\big ],\quad d \in \big [0.275\, , \, 0.416\big ]\quad \mbox{and}\quad c' \in \big [-0.885\, , \, -0.472\big ].
$$
Finally, to visualise the goodness-of-fit of this second long memory LARCH$(\infty)$ model, the correlograms of the residuals and the absolute values of the residuals are plotted in Figure \ref{Fig5}.\\
\begin{figure}[h]
  \hspace*{0cm}
	\centering
  \includegraphics[height=6cm,width=8cm]{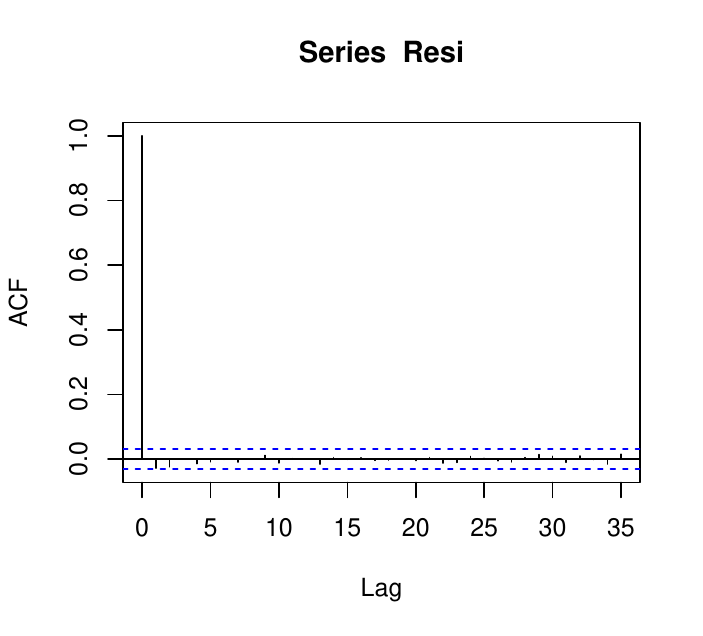} \includegraphics[height=6cm,width=8cm]{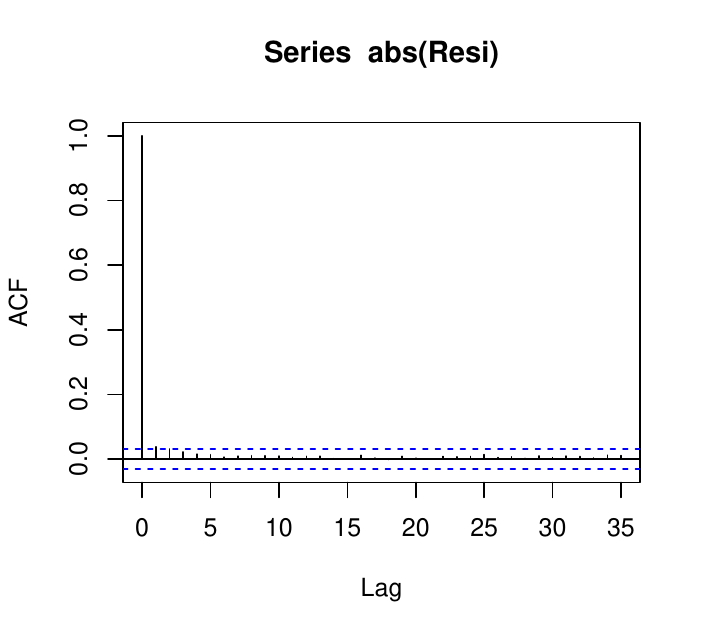}
	\caption{Correlograms of residuals and absolute values of residuals for the returns of CAC40 index, between November 15, 2002 and November 15, 2022, using $a_j(\theta)=c\,j^{d-1}\big (1 +\frac {c'}j \big )$ for $j\geq 1$\label{Fig5}} 
\end{figure}
~\\
{\bf Conclusion:} For these financial data, as observed by other authors, modelling by a long memory LARCH$(\infty)$ process seems relevant. However, we point out that the model studied by Beran and Sch\"utzner (2009), {\it i.e.} $a_j(\theta)=c\, j^{d-1}$ for $j\geq 1$, is not appropriate for these data, while a model with a slightly more complex behaviour with an additional parameter, {\it i.e.} $a_j(\theta)=c\,j^{d-1}\big (1 +\frac {c'}j \big )$ for $j\geq 1$, provides a quite satisfactory fit.   
\section{Proofs} \label{Proofs}
In the sequel, with $\sum_{j=1}^{m}\cdot =0$ for $m\leq 0$ by convention, for $t \in \Z$, we will many times consider 
\begin{equation}\label{Mt}
M_\theta(t):=a_0(\theta)+\sum_{j=1}^\infty a_j(\theta) \, X_{t-j}\quad\mbox{and}\quad \widetilde M_\theta(t):=a_0(\theta)+\sum_{j=1}^{t-1} a_j(\theta) \, X_{t-j}.
\end{equation}
\begin{lem}\label{LemM}
Under the assumptions of Proposition \ref{consLARCH} and with $\overline d<1/2$ defined in \eqref{CondAN}, we obtain:
\begin{enumerate}
\item There exists $\widetilde C>0$, such as for any $t\geq 1$ and $\theta \in \Theta$:
\begin{equation}\label{borneMM}
\E \Big [\big (M_\theta(t)\big )^2 \Big ] <\infty,\quad \E \Big [\big (\widetilde M_\theta(t)\big )^2 \Big ] <\infty \quad \mbox{and} \quad \E \Big [ \big (M_\theta(t)- \widetilde M_\theta(t)\big )^2 \Big ]  \leq \widetilde C \, t^{2\overline d -1}.
\end{equation}
\item There exists $\widetilde C>0$, such as for any $t\geq 1$ and $\theta \in \Theta$:
\begin{equation}\label{bornedMM}
\E \Big [\big \|\partial _\theta M_\theta(t)\big \|^2 \Big ] <\infty,\quad \E \Big [\big \| \partial _\theta \widetilde M_\theta(t)\big \|^2 \Big ] <\infty \quad \mbox{and} \quad \E \Big [\big \| \partial _\theta  M_\theta(t)- \partial _\theta  \widetilde M_\theta(t)\big \|^2  \Big ]  \leq \widetilde C \, t^{2\overline d -1}.
\end{equation}
\item There exists $\overline C>0$, such as for any $t\geq 1$,
\begin{equation} \label{borneM}
\E \Big [\sup_{\theta \in \Theta} \big (M_\theta(t)- \widetilde M_\theta(t)\big )^2 \Big ] \leq 
\overline C \, t^{2\overline d -1}.
\end{equation}
\end{enumerate}
\end{lem}
\begin{proof}[Proof of Lemma \ref{LemM}]
1. Since $(X_t)_{t\in \Z}$ is a weak white noise, it can be written directly that
$$
\E \Big [\big (M_\theta(t)\big )^2 \Big ] =\E [X_0^2] \, \sum_{k=1}^\infty a_k^2(\theta) \leq \E [X_0^2] \, \sum_{k=1}^\infty \sup _{\theta \in \Theta} a_k^2(\theta).
$$
From Assumption $C_\ell(\Theta)$, $\sup _{\theta \in \Theta} \big |a_k(\theta)\big | \leq C_a \, k^{\overline d-1} $ implying 
$$
\sum_{k=1}^\infty \sup _{\theta \in \Theta} a_k^2(\theta) \leq C^2_a \, \sum_{k=1}^\infty k^{2\overline d-2}<\infty,
$$
since $2\overline d-2<-1$. And therefore $\E \Big [\big (M_\theta(t)\big )^2 \Big ] <\infty$. \\
Using the same reasoning, 
$$\E \Big [\big (\widetilde M_\theta(t)\big )^2 \Big ]= \E [X_0^2] \, \sum_{k=1}^{t-1} a_k^2(\theta) \leq \E \Big [\big (M_\theta(t)\big )^2 \Big ] <\infty.
$$
Finally, $\E \Big [ \big (M_\theta(t)- \widetilde M_\theta(t)\big )^2 \Big ]= \E [X_0^2] \, \sum_{k=t}^{\infty} a_k^2(\theta)\leq C^2_a \, \E [X_0^2] \, \sum_{k=t}^{\infty} k^{2\overline d-2} $ always from Assumption $C_\ell(\Theta)$. But for $f$ a positive decreasing function, $\sum_{k=n_0}^{n_1} f(k)\leq f(n_0)+\int_{n_0}^{n_1}f(t)dt$ for any integer numbers $n_0$ and $n_1$ such as $0\leq n_0<n_1\leq \infty$. Therefore, for any $t\geq 1$,
$$
\sum_{k=t}^{\infty} k^{2\overline d-2} \leq t^{2\overline d-2}+\int_t^\infty x^{2\overline d-2}dx\leq t^{2\overline d-2}+\frac 1{1-2\overline d} \, t^{2\overline d-1}\leq \Big (\frac {2-2\overline d}{1-2\overline d}\Big )\, t^{2\overline d-1}.
$$ 
Therefore, for any $\theta \in \Theta$ and $t\geq 1$, 
$$
\E \Big [ \big (M_\theta(t)- \widetilde M_\theta(t)\big )^2 \Big ] \leq C^2_a \, \E [X_0^2] \,\Big (\frac {2-2\overline d}{1-2\overline d}\Big )\, t^{2\overline d-1} \leq \widetilde  C \,t^{2\overline d-1}. 
$$
2. Using Assumptions (S) and  $C_\ell(\Theta)$, $\theta \in \Theta \mapsto \partial _\theta M_\theta(t)$ exists a.s., and for any $\theta \in \Theta$ and $t\geq 1$,
\begin{equation}\label{dM}
\partial _\theta M_\theta(t)=\partial _\theta a_0(\theta)+ \sum_{k=1}^\infty \partial _\theta a_k(\theta) X_{t-k}\quad\mbox{and} \quad  \partial _\theta \widetilde M_\theta(t)=\partial _\theta a_0(\theta)+ \sum_{k=1}^{t-1} \partial _\theta a_k(\theta) X_{t-k}~a.s.
\end{equation}
Then $\E \Big [\big (\partial _\theta M_\theta(t)\big )^2 \Big ]=\sum_{j=1}^\ell \E \Big [\big (\partial _{\theta_j} M_\theta(t)\big )^2 \Big ]$. 
By replacing $a_k(\theta)$ by $\partial _{\theta_j} a_k(\theta)$ that also satisfies \eqref{CondAN} for $j=1,\ldots,\ell$, we can use the same reasoning as in 1. for obtaining $\E \Big [\big (\partial _{\theta_j} M_\theta(t)\big )^2 \Big ]<\infty$. And this can also be done for $\partial_\theta \widetilde M_\theta(t)$ and $\partial_\theta M_\theta(t)-\partial_\theta \widetilde M_\theta(t)$ and we obtain \eqref{bornedMM}. \\
~\\
3. We use here ideas already present in Lemmas 1 and 2 of Beran and Sch\"utzner (2009). The result will be easily generalized, but to facilitate the writing of the proof, we restrict ourselves to the case where $\ell=2$ and $\theta=(\theta_1,\theta_2) \in \Theta \subset [\underline{\theta}_1,\overline{\theta}_1]\times [\underline{\theta}_2,\overline{\theta}_2] $. Denote $Z_\theta=M_\theta(t)- \widetilde M_\theta(t)$.  Using Theorem 3.B of Parzen (1999), we can write that for any $t\geq 1$ and $(\theta_1,\theta_2) \in  [\underline{\theta}_1,\overline{\theta}_1]\times [\underline{\theta}_2,\overline{\theta}_2] $,
$$
Z^2_{(\theta_1,\theta_2)}\leq \frac 1 2 \Big (Z^2_{(\underline{\theta}_1,\theta_2)} +Z^2_{(\overline{\theta}_1,\theta_2)} + \int_{\underline{\theta}_1}^{\overline{\theta}_1} Z^2_{(u_1,\theta_2)}+ \big ( \partial _{\theta_1} Z_{(u_1,\theta_2)} \big )^2\, du_1 \Big ).
$$
By applying again this theorem, and after computations, we finally obtain for any $(\theta_1,\theta_2)\in [\underline{\theta}_1,\overline{\theta}_1]\times [\underline{\theta}_2,\overline{\theta}_2]$,
\begin{multline*}
4\,Z^2_{(\theta_1,\theta_2)}\leq Z^2_{(\underline \theta_1,\underline \theta_2)} +Z^2_{(\underline \theta_1,\overline \theta_2)} + Z^2_{(\overline \theta_1,\underline \theta_2)} +Z^2_{(\overline \theta_1,\overline \theta_2)}\\
+ \int_{\underline{\theta}_2}^{\overline{\theta}_2} \Big ( Z^2_{(\underline \theta_1,u_2)} +Z^2_{(\overline \theta_1,u_2)}+ \big ( \partial _{\theta_2} Z_{(\underline \theta_1,u_2)} \big )^2+ \big ( \partial _{\theta_2} Z_{(\overline \theta_1,u_2)} \big )^2 \Big )\, du_2 \\
+\int_{\underline{\theta}_1}^{\overline{\theta}_1}\Big ( Z^2_{(u_1,\underline \theta_2)}+Z^2_{(u_1,\overline \theta_2)} +\big ( \partial _{\theta_1} Z_{(u_1,\underline \theta_2)} \big )^2 +\big ( \partial _{\theta_1} Z_{(u_1,\overline \theta_2)} \big )^2   \Big )\, du_1 \\
+\int_{\underline{\theta}_1}^{\overline{\theta}_1}\int_{\underline{\theta}_2}^{\overline{\theta}_2}\Big (Z^2_{(u_1,u_2)} + \big ( \partial _{\theta_1} Z_{(u_1,u_2)} \big )^2 +\big ( \partial _{\theta_2} Z_{(u_1,u_2)} \big )^2+\big ( \partial^2 _{\theta_1\theta_2} Z_{(u_1,u_2)} \big )^2  \Big )\, du_1du_2.
\end{multline*}
After taking expectations, we finally obtain that there exist positive real numbers $C_0(\Theta)$, $C_1(\Theta)$, $C_2(\Theta)$ and $C_{12}(\Theta)$ depending only on $\Theta$ such as
\begin{multline}\label{Parzen}
\E \big [\sup_{(\theta_1,\theta_2)\in \Theta}Z^2_{(\theta_1,\theta_2)}  \big ] \leq C_0(\Theta) \,\sup_{(\theta_1,\theta_2)\in \Theta} \E \big [ Z^2_{(\theta_1,\theta_2)}  \big ] + C_1(\Theta) \,\sup_{(\theta_1,\theta_2)\in \Theta} \E \big [ \big (\partial _{\theta_1} Z_{(\theta_1,\theta_2)} \big )^2 \big ] \\
+ C_2(\Theta) \,\sup_{(\theta_1,\theta_2)\in \Theta} \E \big [ \big (\partial _{\theta_2} Z_{(\theta_1,\theta_2)} \big )^2 \big ] +C_{12}(\Theta) \,\sup_{(\theta_1,\theta_2)\in \Theta} \E \big [ \big (\partial^2 _{\theta_1\theta_2} Z_{(\theta_1,\theta_2)}\big )^2  \big ].
\end{multline}
Now, using the previous point 1., we obtain that for any $t\geq 1$:
$$
\sup_{(\theta_1,\theta_2)\in \Theta} \E \big [ Z^2_{(\theta_1,\theta_2)}  \big ]=\sup_{(\theta_1,\theta_2)\in \Theta} \E \big [ \big (M_{(\theta_1,\theta_2)}^t-\widetilde M_{(\theta_1,\theta_2)}^t \big )^2  \big ]\leq \widetilde C \,t^{2\overline d-1}.
$$ 
From point 2., the same bound can be established as well for $\sup_{(\theta_1,\theta_2)\in \Theta} \E \big [ \big (\partial _{\theta_1} Z_{(\theta_1,\theta_2)} \big )^2 \big ] $ and for $\sup_{(\theta_1,\theta_2)\in \Theta}\E \big [ \big (\partial _{\theta_1} Z_{(\theta_1,\theta_2)} \big )^2 \big ]$. From a straightforward  extension to $\partial^2 _{\theta_1\theta_2} Z_{(\theta_1,\theta_2)}$, \eqref{Parzen} implies that there exists $\overline C>0$:
$$
\E \big [\sup_{(\theta_1,\theta_2)\in \Theta}\big (M_{(\theta_1,\theta_2)}^t-\widetilde M_{(\theta_1,\theta_2)}^t \big )^2 \big ] \leq \overline C \,t^{2\overline d-1}.
$$
\end{proof}
\begin{proof}[Proof of Proposition \ref{consLARCH}]
For $\theta \in \Theta$, denote:
\begin{equation} \label{In}
I_n(\theta):=\frac1 n \,  \sum_{t=1}^n \Phi\big ( (X_{t-k})_{k\geq 0} , \theta \big )~~\mbox{and}~~ \widetilde I_n(\theta):=\frac1 n \,  \sum_{t=1}^n \Phi\big ((\widetilde X_{t-k})_{k\geq 0} , \theta \big ),
\end{equation}
where $\Phi$ is defined in \eqref{phiLAV}. The proof will be stepped in 3 points:\\
~\\
{\bf 1.} We prove here that $\sup_{\theta \in \Theta} \big |  I_n(\theta)-I(\theta) \big | \limiteasn 0$, with
\begin{equation} \label{I}
I(\theta):= \E \big [ \Phi\big ( (X_{-k})_{k\geq 0} , \theta \big )\big ]~~\mbox{for $\theta \in \Theta$}.
\end{equation} 
Indeed, from Doukhan and Wintenberger (2008), there exists a function $H: \R^\N \to \R$ such as for any $t\in \Z$, $X_t=H((\xi_{t-j})_{j\geq 0})$ and therefore $(X_t)_{t\in \Z}$ is a second order ergodic stationary sequence since $r=2$. Then, for $\theta \in \Theta$, there exists $H_\Phi:\R^\N \to [0,\infty)$ such that
$$
\Phi\big ( (X_{-k})_{k\geq 0} , \theta \big )=\Phi\big ( \big ((H(\xi_{k-j}))_{j\geq 0}\big )_{k\geq 0} , \theta \big )=H_\Phi \big ( (\xi_{-j})_{j\geq 0}\big ),
$$
with also $\E \big [\big | \Phi\big ( (X_{-k})_{k\geq 0} , \theta \big ) \big |\big ]<\infty$. Then using Theorem 36.4 in Billingsley (1995),  $\big (\Phi\big ( (X_{-k})_{k\geq 0} , \theta \big )\big )_{t\in \Z}$ is an ergodic stationary sequence for any $\theta\in \Theta$. As a consequence, for any $\theta \in \Theta$,
$$
 I_n(\theta) \limiteasn I(\theta).
$$
Moreover, since $\Theta$ is a compact set, using Theorem 2.2.1. in Straumann (2005), we deduce that $\big (\Phi\big ( (X_{t-k})_{k\geq 0} , \theta \big )\big )_{t\in \Z}$ also follows	a uniform ergodic theorem  and we obtain $\sup_{\theta \in \Theta} \big |  I_n(\theta)-I(\theta) \big | \limiteasn 0$.\\
~\\
{\bf 2.} We also have $\sup_{\theta \in \Theta} \big |  I_n(\theta)-\widetilde I_n(\theta) \big | \limiteasn 0$. We have
\begin{eqnarray} \label{inegI}
 \big |  I_n(\theta)-\widetilde I_n(\theta) \big |  & \leq  &\frac 1 n \, \sum_{t=1}^n \big | \Phi\big ( (X_{t-k})_{k\geq 0} , \theta \big ) -\Phi\big ( (\widetilde X_{t-k})_{k\geq 0} , \theta \big ) \big |,
\end{eqnarray}
and for any $\theta \in \Theta$,
\begin{eqnarray*}
\big |\Phi\big ( (X_{t-k})_{k\geq 0} , \theta \big ) -\Phi\big ( (\widetilde X_{t-k})_{k\geq 0} , \theta \big ) \big | & = & \Big | \big (|X_t|-\big|M_\theta(t)\big |\big )^2  -\big (|X_t|-\big|\widetilde M_\theta(t)\big |\big )^2 \Big | \\
& \leq & \big |M_\theta(t)- \widetilde M_\theta(t)\big | \,
 \Big (2 \,  |X_t| + \big |M_\theta(t)\big | + \big |\widetilde M_\theta(t)\big | \Big ),
\end{eqnarray*}
with $M_\theta(t)$ and $\widetilde M_\theta(t)$ defined in \eqref{Mt}. 
Therefore, using Cauchy-Schwarz and Minkowski inequalities,
\begin{eqnarray}
\nonumber \E \Big [\sup_{\theta \in \Theta} \big |\Phi\big ( (X_{t-k})_{k\geq 0} , \theta \big ) -\Phi\big ( (\widetilde X_{t-k})_{k\geq 0} , \theta \big ) \big |\Big ]\quad &&\\
\nonumber & & \hspace{-7.3cm} \leq \E \Big [\sup_{\theta \in \Theta}\Big \{ \big |M_\theta(t)- \widetilde M_\theta(t)\big |\Big \} \,
 \Big (2 \, |X_t|+ \sup_{\theta \in \Theta} \Big \{ \big |M_\theta(t)\big | + \big |\widetilde M_\theta(t)\big | \Big \} \Big ) \Big ] \\ 
\label{pr1} && \hspace{-7.3cm} \leq  \Big ( \E \Big [\sup_{\theta \in \Theta}\Big \{ \big (M_\theta(t)- \widetilde M_\theta(t)\big )^2\Big \} \Big ] \Big ) ^{1/2} \Big (2 \,  \|X_t\|_2 +  \Big (\E \Big [\sup_{\theta \in \Theta} \big |M_\theta(t)\big |^2 \Big ] \Big )^{1/2} \hspace{-3mm}+ \Big (\E \Big [\sup_{\theta \in \Theta} \big |\widetilde  M_\theta(t)\big |^2 \Big ] \Big )^{1/2}  \Big ).~~
\end{eqnarray}
Now,  since $\Theta$ is a compact set included in $\Theta(2)$, we have $\|X_t\|_2 <\infty$. Moreover, using Lemma \ref{LemM}, point 3., for $t=1$, we have $\E \Big [\sup_{\theta \in \Theta} \big |M_\theta(1)\big |^2 \Big ]=\E \Big [\sup_{\theta \in \Theta} \big |M_\theta(t)\big |^2 \Big ] \leq \overline C$. \\
With $\big |\widetilde  M_\theta(t)\big |^2\leq 2 \, \big (M_\theta(t)-\widetilde M_\theta(t) \big )^2 +2 \,\big (M_\theta(t)\big )^2$, this also implies that $\E \Big [\sup_{\theta \in \Theta} \big |\widetilde M_\theta(t)\big |^2 \Big ] \leq 4 \, \overline C$. Therefore,  there exists $C>0$ such as for any $t\geq 1$,
\begin{equation} \label{grs1}
 \Big (2 \,  \|X_t\|_2 +  \Big (\E \Big [\sup_{\theta \in \Theta} \big |M_\theta(t)\big |^2 \Big ] \Big )^{1/2} + \Big (\E \Big [\sup_{\theta \in \Theta} \big |\widetilde  M_\theta(t)\big |^2 \Big ] \Big )^{1/2}  \Big )\leq C. 
\end{equation}
Therefore, from \eqref{pr1}, \eqref{grs1} and Lemma \ref{LemM}, point 3., we deduce that there exists $C>0$ such as:
\begin{eqnarray*} \label{pr2}
 \E \big [\sup_{\theta \in \Theta} \big |\Phi\big ( (X_{t-k})_{k\geq 0} , \theta \big ) -\Phi\big ( (\widetilde X_{t-k})_{k\geq 0} , \theta \big ) \big |\big ] \leq  C \, t^{2\overline d-1} \quad\mbox{for any $t\geq 1$}.
\end{eqnarray*}
Then, using this bound and since $\overline d<1/2$, there exists $C>0$ such as:
\begin{equation}\label{Kou}
\sum_{t=1}^n \frac 1 t \, \E \big [\sup_{\theta \in \Theta} \big |\Phi\big ( (X_{t-k})_{k\geq 0} , \theta \big ) -\Phi\big ( (\widetilde X_{t-k})_{k\geq 0} , \theta \big ) \big |\big ] 
\leq C \, \sum_{t=1}^n  t^{2\overline d-2}\leq C \, \sum_{t=1}^\infty   t^{2\overline d-2}<\infty .
\end{equation}
In Corollary 1 of Kounias and Weng (1969), it is established that $\displaystyle\sum_{t=1}^\infty \frac{\E\big [ |Z_t| \big ]}{b_t} <\infty$ implies $\displaystyle\frac 1 {b_n} \, \sum_{t=1}^n Z_t \limiteasn 0$ for a $\L^1$ sequence of r.v. $(Z_t)_t$ and $b_n\limiten \infty$. Therefore, with $b_t=t$ for $t\in \N^*$, \eqref{Kou} leads to 
$$
\frac 1 n \, \sum_{t=1}^n \sup_{\theta \in \Theta} \big |\Phi\big ( (X_{t-k})_{k\geq 0} , \theta \big ) -\Phi\big ( (\widetilde X_{t-k})_{k\geq 0} , \theta \big ) \big | \limiteasn 0,
$$
and therefore to $\displaystyle \sup_{\theta \in \Theta} \big |  I_n(\theta)-\widetilde I_n(\theta) \big | \limiteasn 0$ from \eqref{inegI}.\\
~\\
{\bf 3.} The two previous points show us that $\sup_{\theta \in \Theta} \big | \widetilde I_n(\theta)- I(\theta)\big | \limiteasn 0$ with $I$ defined in \eqref{I}. The proof is achieved if we establish that $\theta^*$ is the unique minimum of $\theta \in \Theta \mapsto I(\theta)$. This is induced by the following computations:
\begin{eqnarray*}
I(\theta) & =& \E \big [ \Phi\big ( (X_{-k})_{k\geq 0} , \theta \big )\big ] \\
& =& \E \Big [ \Big (|\xi_0| \, \big|a_0(\theta^*)+\sum_{j=1}^{\infty}a_j(\theta^*) \, X_{-j}\big |-\big|a_0(\theta)+\sum_{j=1}^{\infty}a_j(\theta) \, X_{-j}\big |\Big )^2\Big ] \\
& =& \E \big [\xi_0^2-1 \big ]  \, \E \Big [ \Big ( \big|a_0(\theta^*)+\sum_{j=1}^{\infty}a_j(\theta^*) \, X_{-j}\big | \Big )^2 \Big ] \\
&& \hspace{2cm}+    \E \Big [ \Big (\big|a_0(\theta)+\sum_{j=1}^{\infty}a_j(\theta) \, X_{-j}\big | -  \big|a_0(\theta^*)+\sum_{j=1}^{\infty}a_j(\theta^*) \, X_{-j}\big | \Big )^2 \Big ], 
\end{eqnarray*}
using the assumption $\|\xi_0\|_1=1$ and because $(X_t)$ is a causal time series implying that $\xi_0$ independent to $\sigma \big \{ (X_{-k})_{k\geq 1} \big \}$. The first term of the previous relationship does not depend on $\theta$. The second one vanishes when $\theta = \theta^*$. It is also non negative and it vanishes if 
$$
\big|a_0(\theta)+\sum_{j=1}^{\infty}a_j(\theta) \, X_{-j}\big | = \big|a_0(\theta^*)+\sum_{j=1}^{\infty}a_j(\theta^*) \, X_{-j}\big |\quad a.s.         
$$
As we assumed that $a_0(\cdot)$ is a positive function, using also Assumption Id$(\Theta)$, we deduce that $\theta=\theta^*$ is the only solution of the previous equality. As a consequence, $\theta^*$ is the unique minimizer of $I(\cdot)$ and since $\sup_{\theta \in \Theta} \big | \widetilde I_n(\theta)- I(\theta)\big | \limiteasn 0$ and $\displaystyle \widehat \theta_n = \widehat \theta_n=\mbox{Arg}\! \min_{\! \! \! \!\theta \in \Theta} \widetilde I_n(\theta)$, we deduce that $\widehat \theta_n \limiteasn \theta^*$.
\end{proof}

\begin{proof}[Proof of Corollary \ref{consLARCHp}]
In such a case, $\theta={}^t(a_0,\ldots,a_p)$ and therefore $a_i(\theta)=a_i$ for any $0\leq i \leq p$ and $a_i(\theta)=0$ for $i\geq p+1$. Then, $\theta \in \Theta \mapsto a_k(\theta)$ are ${\cal C}^\infty$ functions on $\Theta$ for any $k\in \N$. Moreover, since $\sup_{\theta \in \Theta}|a_k(\theta)|=0$ for $k\geq p+1$ as well as $\sup_{\theta \in \Theta}\big \|\partial^j_{\theta^j} a_k(\theta)\big \|=0$ for any $1\leq j\leq p+1$, therefore \eqref{CondAN} is satisfied and Assumption $C_{p+1}(\Theta)$ holds. Since $\sum_{k=1}^\infty\sup_{\theta \in \Theta}|a_k(\theta)| <\infty$, Assumption (S) is also satisfied.
Finally,  Assumption Id$(\Theta)$ is obviously satisfied: $a_i(\theta)=a_i(\widetilde  \theta)$ implies $a_i=\widetilde a_i$, which implies $\theta={}^t(a_0,\ldots,a_p)=\widetilde  \theta={}^t(\widetilde  a_0,\ldots,\widetilde  a_p)$ for any $i\in \N$. Then the strong consistency of $\widehat \theta_n$ is established from Proposition \ref{consLARCH}.
\end{proof}
\begin{proof}[Proof of Corollary \ref{consGLARCH}]
Denote $\theta=\big (\theta_P,\theta_Q\big ) \in \Theta$ where $\theta_P={}^t \big ( c_0,c_1,\ldots,c_p\big ) \in (0,\infty)\times \R^{p}$, $\theta_Q={}^t \big (d_1,\ldots,d_q\big )\in \R^{q}$. Then $\sigma=P^{-1}_{\theta_P}(B) \, Q_{\theta_Q}(B) X$. It is clear that $\theta_P \to P_{\theta_P}$ is an injective function, and it is the same for $\theta_P \to P^{-1}_{\theta_P}$ and $\theta_Q \to  Q_{\theta_Q}$. Finally it is also the same for $\theta={}^t\big (\theta_P,\theta_Q \big ) \to P^{-1}_{\theta_P}\times  Q_{\theta_Q}$ , because $P_{\theta_P}$ and $Q_{\theta_Q}$ are not zero polynomial and because we assume that the $p+1$ components of $\theta_P$ are free of the $q$ components $\theta_Q$, {\it i.e.} there are no supposed links between  $(c_i)_{0\leq i\leq p} $ and $(d_j)_{1\leq j \leq q}$. Therefore, \eqref{identLARCH} and Assumption Id$(\Theta)$ are satisfied for GLARCH$(p,q)$ process. \\
Moreover, as it is well known for GARCH$(p,q)$ processes, since $\sum_{j=1}^q |d_j|\leq \rho$ with $0\leq \rho<1$ from the expression of $\Theta_{p,q}(2)$, the roots $z_j(\theta)$ of the characteristic polynomial $\xi(z)=z^q-\sum_{j=1}^qb_jz^{q-j}$ satisfies $\sup_\theta \max_j |z_j(\theta)|<1$: $(a_k(\theta))_k$ decreases exponentially fast towards $0$. Then $\sum_{k=1}^\infty \sup_{\theta \in \Theta}|a_k(\theta)|<\infty$ and Assumption (S) is satisfied. Moreover,  there exists $c_0>0$ such as $\sup_{\theta \in \Theta}|a_k(\theta)|\leq c_0\, k^{-3/4}$ for any $k\geq 1$. By considering the derivatives of equation \eqref{GLARCH}, we also have $\sup_{\theta \in \Theta}\big \|\partial^j _{\theta^j} a_k(\theta)\big \|\leq c_j\, k^{-3/4}$ with $c_j>0$ for any $k,j\geq 1$ and  Assumption $C_{p+q+1}(\Theta)$ holds.   
\end{proof}
\begin{proof}[Proof of Corollary \ref{consLRD}]
The assumptions of Corollary \ref{consLRD} are exactly the same as those of Proposition \ref{consLARCH}. It remains for us to prove that \eqref{CondAN} could be satisfied under the condition \eqref{LRD}, {\it i.e.} there exists $d(\theta)\in (0,1/2)$ and $L_\theta(\cdot)$ a slowly varying function such that  $a_j(\theta)=L_\theta(j)\, j^{d(\theta)-1}$ for $j \in \N^*$. Since $\theta \in \Theta$ a compact set, there exists $D \in (0,1/2)$ such that $d(\theta)\leq D$ for any $\theta \in \Theta$. Moreover, since $L_\theta(\cdot)$ is a slow varying function and $\theta \in \Theta$ a compact subset of $\R^\ell$, there exists $C_L>0$ such that $\sup_{\theta \in \Theta}\big |L_\theta(j)\big | \leq C_L \, j^{1/4-D/2}$ for any $j\in \N^*$, with $1/4-D/2>0$. As a consequence,
$$
\sup_{\theta \in \Theta} \big | a_j(\theta) \big | \leq C_L\, j^{1/4-D/2} j^{D-1} \leq C_L\, j^{D/2-3/4} \quad \mbox{for any $j\in \N^*$}.
$$
With $\overline d=D/2+1/4<1/2$, Assumption $C_0(\Theta)$ is verified and Corollary \ref{consLRD} is established.
\end{proof}
\begin{proof}[Proof of Corollary \ref{corBeran}]
We have to prove that the assumptions of  Corollary \eqref{consLRD} are satisfied in this particular case of long memory LARCH$(\infty)$ process.\\
First, we have $\theta={}^t(a_0,c,d) \in \Theta \subset \R^3$ and Assumption $C_3(\Theta)$ has to be verified. It is clear that $a_k:\theta \mapsto c\, k^{d-1}$ is a ${\cal C}^\infty$ function on $\Theta$, with $\theta$ defined in \eqref{ThetaLRD2}. For $k\geq 3$, easy computations imply that:
\begin{equation*}
\sup_{\theta \in \Theta} \big \|\partial^j _{\theta^j} a_k(\theta)\big \| \leq  \sqrt{c_M^2+j}\times  \log k \times k^{d_M-1} \quad \mbox{for $1\leq j\leq 3$}.
\end{equation*}
Therefore, by considering for instance $\overline d=d_M/2+1/2 \in (d_M,1/2)$, it is clear that there exists $C_a>0$ such as \eqref{CondAN} is satisfied and Assumption $C_3(\Theta)$ holds.
~\\
It remains to prove Assumption Id$(\Theta)$. This one will be verified by considering the equality $a_j(\theta)=a_j(\widetilde \theta)$ for any $j\in \N$ where $\theta={}^t(a_0,c,d)$ and $\widetilde \theta={}^t(\widetilde a_0,\widetilde c,\widetilde d)$. This implies $a_0=\widetilde  a_0$ and $c \, j^{d-1}=\widetilde  c\, j^{\widetilde d-1}$ for any $j\in \N$, leading to $\theta=\widetilde \theta$: Assumption Id$(\Theta)$ is also satisfied. 
\end{proof}
\begin{proof}[Proof of Theorem \ref{ANLARCH}]
Let $I_n(\theta)$ and $\widetilde I_n(\theta)$ be defined in \eqref{In}.  We follow a proof that is similar to the one of Theorem 2 in Davis and Dunsmuir (1997). \\
Let $v =\sqrt n (\theta-\theta^*) \in \R^\ell$ and define 
\begin{eqnarray*} 
W_n(v)&=&  \sum_{t=1}^n \Phi\big ( (X_{t-k})_{k\geq 0} , \theta^*+n^{-1/2} v\big ) -\Phi\big ( (X_{t-k})_{k\geq 0} , \theta^*\big ) =n\,\big (I_n(\theta)-I_n(\theta^*) \big)\\
\mbox{and}\quad \widetilde W_n(v)&=&   \sum_{t=1}^n \Phi\big ( (\widetilde X_{t-k})_{k\geq 0} , \theta^*+n^{-1/2} v\big ) -\Phi\big ( (\widetilde X_{t-k})_{k\geq 0} , \theta^*\big )= n\,\big (  I_n(\theta)- \widetilde I_n(\theta^*) \big).
\end{eqnarray*}
Then we are going to prove first that minimizing $\widetilde I_n(\theta)$ with respect to $\theta \in \Theta$ is equivalent to minimize $\widetilde W_n(v)$ with respect to $v \in \R^\ell$, which is also equivalent to minimize $W_n(v)$ with respect to $v \in \R^\ell$. %As a consequence, if there exists a unique minimum in $\R^\ell$ of $W_n(v)$, there exists a sequence  $(\widehat v_n)_n$ where $\widehat v_n$ is a minimizer of $W_n(v)$. And by writting $\widehat v_n=\sqrt n (\widehat \theta_n-\theta^*)$, $\widehat \theta_n$ is a unique minimum of . 
Secondly, we will provide a limit theorem satisfied by $W_n(v)$ for any $v \in \R^\ell$.   Then we are going to prove in 3. that $(W_n(\cdot ))_n$ converges as a process of ${\cal C}(\R^\ell)$ (space of continuous functions on $\R^\ell$) to a limit process $W$. Hence the sequence of minimum of  $\widetilde W_n$, {\it i.e} $(\widehat v_n)_n$ with $\widehat v_n=\sqrt n (\widehat \theta_n-\theta^*)$, will converge in distribution to the distribution of the minimum of $W(\cdot)$. \\
~\\
1.  For any $v \in \R^\ell$ and $n\geq 1$, we have:
\begin{eqnarray} 
\nonumber
W_n(v)
 &=& \sum_{t=1}^n \Big (\big |X_t\big |-\big |M_{\theta^*+n^{-1/2} v}(t)\big |\Big )^2-\Big (\big |X_t\big |-\big | M_{\theta^*}(t)\big |\Big )^2 \\
\nonumber &=&  \sum_{t=1}^n \Big (\big |X_t\big |-\Big |M_{\theta^*}(t)-\frac 1 {\sqrt n} \, {}^t v \,\partial _\theta M_{\overline \theta^{(n)}_t}(t) )\Big |\Big )^2-\Big (\big |X_t\big |-\big |M_{\theta^*}(t)\big |\Big )^2 \\
\nonumber &=&  \sum_{t=1}^n \Big (\big |X_t\big |-\big |M_{\theta^*}(t)\big |-\frac 1 {\sqrt n} \, {}^t v \,\partial _\theta M_{\overline \theta^{(n)}_t}(t)\times \mbox{sgn}(M_{\theta^*}(t) )\Big )^2-\Big (\big |X_t\big |-\big |M_{\theta^*}(t)\big |\Big )^2 \\
\nonumber &=& -\frac 2 {\sqrt  n} \, \sum_{t=1}^n \big (\big |X_t\big | -\big |M_{\theta^*}(t)\big |\big )\times \mbox{sgn}(M_{\theta^*}(t) )\times  {}^t v \,\partial _\theta M_{\overline \theta^{(n)}_t}(t)+ \frac 1 n \, \sum_{t=1}^n  \big ({}^t v \,\partial _\theta M_{\overline \theta^{(n)}_t}(t) )\big )^2 \\
\label{J1J2}&=& J_1^{(n)}(v)+J_2^{(n)}(v)
\end{eqnarray}
with $\overline \theta^{(n)}_t=\alpha^{(n)}_t \, \theta^*+(1-\alpha^{(n)}_t )\big ( \theta^*+ n^{-1/2} v \big )$ where $\alpha^{(n)}_t \in [0,1]$ is given from the Taylor-Lagrange expansion. \\
~\\
{\bf Term $J_2^{(n)}(v)$:} For any $v \in \R^\ell$, $\overline \theta^{(n)}_t \limiteasn \theta^*$ and then 
\begin{equation}\label{MM}
\Big |\partial _\theta M_{\overline \theta^{(n)}_t}(t)-\partial _\theta M^t_{\theta^*)} \Big | \limiteasn 0\quad \mbox{for any $t\in \N$}, 
\end{equation}
since the functions $\theta \in \Theta \mapsto \partial _\theta a_i(\theta)$ are supposed to be continuous functions for any $i \in \N$. Then we obtain for any $v \in \R^\ell$:
\begin{multline*}
\Big |J_2^{(n)}(v)-\E \big [ \big ({}^t v \,\partial _\theta M_{\theta^*}(0) \big )^2 \big ]  \Big |\leq \frac 1 n \, \sum_{t=1}^n  \Big |\big ({}^t v \,\partial _\theta M_{\theta^*}(t)\big )^2- \big ({}^t v \,\partial _\theta M_{\overline \theta^{(n)}_t}(t)\big )^2\Big | \\
+  \Big |\frac 1 n \, \sum_{t=1}^n  \big ({}^t v \,\partial _\theta M_{\theta^*}(t)\big )^2-\E \big [ \big ({}^t v \,\partial _\theta M_{\theta^*}(0) \big )^2 \big ]  \Big |. 
\end{multline*}
Now using Cesaro Lemma we obtain from \eqref{MM}, 
\begin{equation}\label{Cesar}
\Big |\frac 1 n \, \sum_{t=1}^n  \big ({}^t v \,\partial _\theta M_{\theta^*}(t)\big )^2-\E \big [ \big ({}^t v \,\partial _\theta M_{\theta^*}(0) \big )^2 \big ]  \Big |\limiteasn 0. 
\end{equation}
Moreover,  we have seen that there exist a function $H$ such as $X_t=H((\xi_{t-j})_{j\geq 0})$ for any $t\in \Z$, and therefore $(X_t)_{t\in \Z}$ is a second order ergodic stationary sequence. Then, for any $v \in \R^\ell$, there exists a function $H_v:\R^\N \to [0,\infty)$ such as
$$
\big ({}^t v \,\partial _\theta M_{\theta^*}(t)\big )^2=H_v \big ( (\xi_{-j})_{j\geq 0}\big ),
$$
with also $\E \big [\big ({}^t v \,\partial _\theta M_{\theta^*}(t)\big )^2\big ]<\infty$ (see 1. of Lemma \ref{LemM}). Then using Theorem 36.4 in Billingsley (1995),  $\big (\big ({}^t v \,\partial _\theta M_{\theta^*}(t)\big )^2\big )_{t\in \Z}$ is an ergodic stationary sequence implying to:
\begin{equation}\label{Ergo}
\frac 1 n \, \sum_{t=1}^n  \Big |\big ({}^t v \,\partial _\theta M_{\theta^*}(t)\big )^2- \E \big [ \big ({}^t v \,\partial _\theta M_{\theta^*}(0) \big )^2 \big ] \Big | \limiteasn 0. 
\end{equation}
Finally, with \eqref{Cesar} and \eqref{Ergo}, we obtain for any $v \in \R^\ell$,
\begin{equation}\label{LimJ2}
J_2^{(n)}(v) \limiteasn \E \big [ \big ({}^t v \,\partial _\theta M_{\theta^*}(0) \big )^2 \big ] = {}^t v \, \Gamma^*_1 \, v\quad \mbox{where} \quad \Gamma_1^* :=\E \Big [ \partial _\theta M_{\theta^*}(0)  \times {}^t  \partial _\theta M_{\theta^*}(0)\Big ],
\end{equation}
where $\Gamma_1^*$ is as in  \eqref{Gamma1}. \\
~\\
{\bf Term $J_1^{(n)}(v)$:} We also have 
\begin{eqnarray*}
J_1^{(n)}(v)&=&-\frac 2 {\sqrt  n} \, \sum_{t=1}^n \big (\big |M_{\theta^*}(t) \, \xi_t\big | -\big |M_{\theta^*}(t)\big |\big ) \times \mbox{sgn}(M_{\theta^*}(t) ) \times  {}^t v \,\partial _\theta M_{\overline \theta^{(n)}_t}(t) \\
&=& {}^t v \, \Big (-\frac 2 {\sqrt  n} \, \sum_{t=1}^n \big (\big |\xi_t\big | -1 \big )\, M_{\theta^*}(t) \times \partial _\theta M_{\theta^*}(t) \\
&& \hspace{3cm}+\frac 2 {n} \, \sum_{t=1}^n \big (\big |\xi_t\big | -1 \big )\, M_{\theta^*}(t) \times  \sqrt n \, \big 
(\partial _\theta M_{\theta^*}(t)- \partial _\theta M_{\overline \theta^{(n)}_t}(t) \big ) \Big ) \\
&=& {}^t v \,\big (K^{(n)}_1(v)+K^{(n)}_2(v)\big ).
\end{eqnarray*}
We have $\Big (\big (\big |\xi_t\big | -1 \big )\, M_{\theta^*}(t) \times  \partial _\theta M_{\theta^*}(t) \Big)_{t\in \N}$ that is a stationary ergodic martingale difference since with the $\sigma$-algebra ${\cal F}_t =\sigma\big \{ (X_{t-k})_{k\geq 1} \big \}$,
$$
\E \Big [ \big (\big |\xi_t\big | -1 \big )\, M_{\theta^*}(t) \times \partial _\theta M_{\theta^*}(t)  ~\Big | \, {\cal F}_t \Big ]=\E \big [ \big |\xi_t\big | -1  \big ] \, \E \Big [ M_{\theta^*}(t) \times \partial _\theta M_{\theta^*}(t)  \Big ]=0,
$$
because $(X_t)$ is a causal process and $\xi_t$ is independent of ${\cal F}_t$ and $\E \big [ \big | \xi_0 \big | \big ]=1$.  \\
Now since  $\Gamma_2^* :=\E \Big [\big (M_{\theta^*}(0) \big )^2  \partial _\theta M_{\theta^*}(0)  \times {}^t  \partial _\theta M_{\theta^*}(0)\Big ]$ is supposed to be a finite definite positive matrix (see also its expression in \eqref{Gamma2}), 
\begin{equation*}
\E \Big [ \big (\big |\xi_0\big | -1 \big )^2 \big  \| M_{\theta^*}(0) \times \partial _\theta M_{\theta^*}(0) \big \|^2 \Big ] 
= (\sigma_\xi^2 -1) \, \E \Big [ \big \|M_{\theta^*}(0) \times  \partial _\theta M_{\theta^*}(0)\big \|^2\Big ]  <\infty. 
\end{equation*} 
Then the central limit for stationary ergodic martingale difference, Theorem 18.3 of Billingsley (1968)  can be applied and we obtain for any $v \in \R^\ell$:
\begin{equation} \label{limK1}
K_1^{(n)}(v) \limiteloin K_1 \egaleloi {\cal N} \big ( 0 \, , \, 4 \,(\sigma_\xi^2 -1) \,  \Gamma_2^*  \big ). 
\end{equation} 
For any $t\in \Z$, by the  definition of $\overline \theta^{(n)}_t$, we have $\sqrt n \,(\overline \theta^{(n)}_t-\theta^*)=(1-\alpha_t^{(n)})\, v$ with $\alpha_t^{(n)}$ a random variable of ${\cal F}_t$ bounded in $[0,1]$. Therefore, using a Taylor-Lagrange expansion, we also have
\begin{equation*}
\Big |\sqrt n \, {}^t v \,\big 
(\partial _\theta M_{\theta^*}(t)- \partial _\theta M_{\overline \theta^{(n)}_t}(t) \big ) \Big | =(1-\alpha_t^{(n)}) {}^t v \, \partial _{\theta^2} M^t_{\overline {\overline \theta}^{(n)}_t} \, v 
\leq  \sup_{\theta \in \Theta} \big \| \partial^2 _{\theta^2} M_{\theta}(t)\big \|  \, \big \| v \big \|^2 \quad\mbox{for any $t\in \Z$,}
\end{equation*}
where $\overline {\overline \theta}^{(n)}_t=\beta_t^{(n)} \theta^* + (1-\beta_t^{(n)} )\overline \theta^{(n)}_t$ and $\beta_t^{(n)} \in [0,1]$. \\
We have $ \E \Big [ \Big (  \sup_{\theta \in \Theta} \big \| \partial^2 _{\theta^2} M_{\theta}(t)\big \| \Big )^2 \Big ]= \E \Big [ \sup_{\theta \in \Theta} \big \| \partial^2 _{\theta^2} M_{\theta}(0)-\partial^2 _{\theta^2} \widetilde M_{\theta}(0)\big \|^2 \Big ]$. Using the same reasoning than in 3. of Lemma \ref{LemM}, $\E \Big [ \sup_{\theta \in \Theta} \big \| \partial^2 _{\theta^2} M_{\theta}(0)-\partial^2 _{\theta^2} \widetilde M_{\theta}(0)\big \|^2 \Big ]<\infty$ when Assumption $C_{\ell+2}(\Theta)$ holds. As a consequence,
$$ \E \Big [ \Big (\sqrt n \, {}^t v \, \big 
(\partial _\theta M_{\theta^*}(t)- \partial _\theta M_{\overline \theta^{(n)}_t}(t) \big ) \Big )^2 \Big ] \leq  \E \Big [ \sup_{\theta \in \Theta} \big \| \partial^2 _{\theta^2} M_{\theta}(t)\big \|^2 \Big ] \times \big \| v \big \|^4<\infty.
$$
Moreover, since $\overline \theta^{(n)}_t\in {\cal F}_t$, we also have $(\partial _\theta M_{\theta^*}(t)- \partial _\theta M_{\overline \theta^{(n)}_t}(t) \big )  \in {\cal F}_t$ and from the previous bound,
\begin{multline}
\E \Big [ \Big | \big (|\xi_t| -1 \big )\, M_{\theta^*}(t) \times  \sqrt n \, {}^t v \big 
(\partial _\theta M_{\theta^*}(t)- \partial _\theta M_{\overline \theta^{(n)}_t}(t) \big )  \Big |\Big ] \\
=\E \big [ \big ||\xi_t| -1 \big |\big ] \times  \E \Big [ \Big | \big (\, M_{\theta^*}(t) \times  \sqrt n \, {}^t v \big 
(\partial _\theta M_{\theta^*}(t)- \partial _\theta M_{\overline \theta^{(n)}_t}(t) \big )  \Big |\Big ] \\
\leq \E \big [ \big ||\xi_t| -1 \big |\big ] \times  \E \Big [  \big (M_{\theta^*}(t) \big )^2 \Big ] \times \E \Big [\Big ( \sqrt n \, {}^t v \big 
(\partial _\theta M_{\theta^*}(t)- \partial _\theta M_{\overline \theta^{(n)}_t}(t) \big )  \Big )^2\Big ]<\infty. \label{mom1}
\end{multline}
Therefore $\big (|\xi_t| -1 \big )\, M_{\theta^*}(t) \times  \sqrt n \, {}^t v \big 
(\partial _\theta M_{\theta^*}(t)- \partial _\theta M_{\overline \theta^{(n)}_t}(t) \big )_t$ is a stationary causal sequence, and from \eqref{mom1}, Theorem 36.4 in Billingsley (1995)  implies  that for any $v \in \R^\ell$, 
\begin{equation}\label{K2}
{}^t v  \,K_2^{(n)}(v) \limiteasn \E \Big [\big (|\xi_t| -1 \big )\, M_{\theta^*}(t) \times  \sqrt n \, {}^t v \big 
(\partial _\theta M_{\theta^*}(t)- \partial _\theta M_{\overline \theta^{(n)}_t}(t) \big ) \Big ] =0.
\end{equation}
Finally, for any $v\in \R^\ell$, since $J_1^{(n)}(v)={}^t v  \, \big (K_1^{(n)}(v)+K_2^{(n)}(v)\big )$,   then $\displaystyle  J_1^{(n)}(v) \limiteloin {}^t v  \, K_1$ from \eqref{limK1} and \eqref{K2}, 
and with \eqref{LimJ2} this implies,
\begin{equation} \label{limWn}
W_n(v) \limiteloin {}^t v \, \Gamma^*_1 \, v  + {}^t v  \, K_1 \quad \mbox{with}\quad K_1\egaleloi {\cal N} \big ( 0\, , \, 4 \,(\sigma_\xi^2 -1) \,  \Gamma_2^*  \big ).
\end{equation} 
~\\
2. Asymptotically, for any $v\in \R^\ell$, from 1., we know that the law of $W_n(v)$ is the same as the law of:
$$
W'_n(v)=-\frac 2 {\sqrt  n} \, \sum_{t=1}^n \big (\big |\xi_t\big | -1 \big )\, M_{\theta^*}(t) \times  {}^t v \,\partial _\theta M_{\theta^*}(t)+\frac 1 n \, \sum_{t=1}^n  \big ({}^t v \,\partial _\theta M_{\theta^*}(t)\big )^2
$$
and we deduce the same kind of result for the law of $\widetilde W_n(v)$, which is asymptotically equivalent to the one of:
$$
\widetilde W'_n(v)=-\frac 2 {\sqrt  n} \, \sum_{t=1}^n \big (\big |\xi_t\big | -1 \big )\, \widetilde M_{\theta^*}(t) \times  {}^t v \,\partial _\theta \widetilde  M_{\theta^*}(t)+\frac 1 n \, \sum_{t=1}^n  \big ({}^t v \,\partial _\theta \widetilde  M_{\theta^*}(t)\big )^2.
$$
Therefore we obtain:
\begin{eqnarray}
\nonumber W'_n(v)-\widetilde W'_n(v)&=&-\frac 2 {\sqrt  n} \, \sum_{t=1}^n \big (\big |\xi_t\big | -1 \big )\, {}^t v \, \Big (  M_{\theta^*}(t) \, \partial _\theta M_{\theta^*}(t)-\widetilde M_{\theta^*}(t) \,\partial _\theta \widetilde  M_{\theta^*}(t) \Big )\\ 
&& \hspace{0.5cm} + {}^t v \, \Big ( \frac 1 n \, \sum_{s=1}^n \big ( \partial _\theta  M^s_{\theta^*} {}^t \partial _\theta  M^s_{\theta^*}- \partial _\theta \widetilde  M^s_{\theta^*} {}^t \partial _\theta \widetilde  M^s_{\theta^*} \big )\Big ) \, v \\ 
\label{WW}&& =L^{(n)}_1+L^{(n)}_2.
\end{eqnarray}
{\bf Term $L^{(n)}_1$:}  Using the causality of $(X_t)$, {\it i.e.} $\xi_t$ independent to $\sigma\{X_{t-1},X_{t-2},\ldots \}$ for any $t\in \Z$, we deduce that:
\begin{equation}\label{BC1}
\E \Big [ \frac 1 {\sqrt  n} \,  \sum_{t=1}^n \big (\big |\xi_t\big | -1 \big )\, {}^t v \, \big (  M_{\theta^*}(t) \, \partial _\theta M_{\theta^*}(t)-\widetilde M_{\theta^*}(t) \,\partial _\theta \widetilde  M_{\theta^*}(t) \big ) \Big ]=0,
\end{equation}
since $\E [|\xi_t|]=1$. Moreover, we have for any $t \geq 1$:
\begin{equation*}
\nonumber M_{\theta^*}(t) \, \partial _\theta M_{\theta^*}(t)-\widetilde M_{\theta^*}(t) \,\partial _\theta \widetilde  M_{\theta^*}(t)= M_{\theta^*}(t) \,\big (  \partial _\theta M_{\theta^*}(t)-\partial _\theta \widetilde  M_{\theta^*}(t) \big ) +\partial _\theta \widetilde  M_{\theta^*}(t) \, \big (M_{\theta^*}(t) -\widetilde M_{\theta^*}(t) \big ). 
\end{equation*} 
Using 1. and 2. of Lemma \ref{LemM}, we have for any $t \geq 1$:
\begin{eqnarray*}
\E \big [ \big (M_{\theta^*}(t)-\widetilde M_{\theta^*}(t) \big ) ^2 \big ]\leq  \widetilde C \, t^{2\overline d -1}\quad\mbox{and}\quad
\E \big [ \big \| \partial _\theta M_{\theta^*}(t)-\partial _\theta \widetilde  M_{\theta^*}(t)\big \| ^2 \big ]\leq \widetilde C \, t^{2\overline d -1}.
\end{eqnarray*}
As a consequence, using Cauchy-Schwarz and Minkowski inequalities,
\begin{equation}
\label{MdM} \big \|  M_{\theta^*}(t) \, \partial _\theta M_{\theta^*}(t)-\widetilde M_{\theta^*}(t) \,\partial _\theta \widetilde  M_{\theta^*}(t) \big \|_2  \leq  \sqrt{\widetilde C} \, \big (  \big \| M_{\theta^*}(0) \|_2   + \big \|\partial _\theta  M_{\theta^*}(0) \|_2   \big )\, t^{\overline d -1/2}.
\end{equation}
Therefore, there exists $C>0$ such as for any $n\geq 1$,
\begin{eqnarray*}
\E \Big [ \Big (\frac 1 {\sqrt  n} \,  \sum_{t=1}^n \big (\big |\xi_t\big | -1 \big )\, {}^t v \, \big (  M_{\theta^*}(t) \, \partial _\theta M_{\theta^*}(t)-\widetilde M_{\theta^*}(t) \,\partial _\theta \widetilde  M_{\theta^*}(t) \big ) \Big )^2 \Big ]&& \\
&& \hspace{-8cm} = \frac 1 n \,(\sigma^2_\xi-1) \,\|v\|^2 \, \sum_{t=1}^n   \big \|  M_{\theta^*}(t) \, \partial _\theta M_{\theta^*}(t)-\widetilde M_{\theta^*}(t) \,\partial _\theta \widetilde  M_{\theta^*}(t) \big \|^2_2 \\
&& \hspace{-8cm} \leq  \frac {C}  n \,(\sigma^2_\xi-1) \,\|v\|^2 \, \sum_{t=1}^n t^{2\overline d -1} \\
&& \hspace{-8cm} \leq \frac {C}  n \,(\sigma^2_\xi-1) \,\|v\|^2 \,\Big ( 1+  \int_{1}^n x^{2\overline d -1} dx \Big )\leq C' \,\|v\|^2  \, n^{2\overline d-1} 
\end{eqnarray*}
with $C'>0$ and for $n$ large enough, 
using again the fact that for $f$ a positive decreasing function $\sum_{k=1}^{n} f(k)\leq f(1)+\int_{1}^{n}f(t)dt$. Therefore, for any $v \in \R$, 
\begin{equation}\label{BC2}
\E \Big [ \Big (\frac 1 {\sqrt  n} \,  \sum_{t=1}^n \big (\big |\xi_t\big | -1 \big )\, {}^t v \, \big (  M_{\theta^*}(t) \, \partial _\theta M_{\theta^*}(t)-\widetilde M_{\theta^*}(t) \,\partial _\theta \widetilde  M_{\theta^*}(t) \big ) \Big )^2 \Big ]\limiten 0.
\end{equation}
Therefore, from \eqref{BC1} and \eqref{BC2}, we deduce $\E\big[L^{(n)}_1\big]=0$ and $\var\big (L^{(n)}_1\big )\limiten 0$. Using Bienaym\'e-Tchebytchev Inequality, this implies that for any $v\in \R^\ell$,
\begin{equation}\label{Mart1}
L^{(n)}_1 \limiteproban 0.
\end{equation}
{\bf Term $L^{(n)}_2$:} Using the same method, we also obtain that there exist
 $C''>0$ such as for any $s \in \{1,\ldots,n\}$,
\begin{equation}\label{dMdM} 
\big \| \partial _\theta M^s_{\theta^*} \, {}^t \partial _\theta M^s_{\theta^*}- \partial _\theta \widetilde M^s_{\theta^*} \,{}^t \partial _\theta \widetilde  M^s_{\theta^*} \big \|_2  \leq C'' \,s^{\overline d-1/2}.
\end{equation}
Now with \eqref{dMdM},  we can use again the result established in part 2. of the proof of Proposition \ref{consLARCH} based on the Corollary 1 of Kounias and Weng (1969) and obtain for any $n \in \N^*$,
$$
\sum_{s=1}^n \frac 1 s \, \big \| \partial _\theta M^s_{\theta^*} \, {}^t \partial _\theta M^s_{\theta^*}- \partial _\theta \widetilde M^s_{\theta^*} \,{}^t \partial _\theta \widetilde  M^s_{\theta^*} \big \|_1 \leq C'' \, \sum_{s=1}^n s^{\overline d-3/2}<\infty
$$
since $\overline d <1/2$. Therefore:
\begin{equation} \label{titi}
\frac 1 n \, \sum_{s=1}^n \big ( \partial _\theta  M^s_{\theta^*} {}^t \partial _\theta  M^s_{\theta^*}- \partial _\theta \widetilde  M^s_{\theta^*} {}^t \partial _\theta \widetilde  M^s_{\theta^*} \big ) \limiteasn 0,
\end{equation}
implying 
\begin{equation}\label{limlim1}
L^{(n)}_2 \limiteasn 0 \qquad\mbox{for any $v\in \R^\ell$.} 
\end{equation}
Finally from \eqref{WW}, \eqref{Mart1} and \eqref{limlim1}, we deduce that for any $v \in \R^\ell$,
\begin{equation*}
\big | W_n(v)- \widetilde W_n(v) \big | \limiteproban 0.
\end{equation*}
Using \eqref{limWn}, this implies
\begin{equation}\label{WWW} 
\widetilde W_n(v)  \limiteloin W(v):={}^t v \, \Gamma^*_1 \, v  + {}^t v  \, K_1 \quad \mbox{with}\quad K_1\egaleloi {\cal N} \big ( 0\, , \, 4 \,(\sigma_\xi^2 -1) \,  \Gamma_2^*  \big ).
\end{equation}
3. Now, using the same arguments than in the proof of Theorem 2 of Davis and Dunsmuir (1997), we deduce that finite distributions $(\widetilde  W_n(v_1),\cdots,\widetilde W_n(v_k))$ converge to $(W(v_1),\cdots,W(v_k))$ for any $(v_1,\cdots,v_k)\in (\R^\ell)^k$. Moreover, always following the proof of Theorem 2 of Davis and Dunsmuir (1997), $(W_n(v))_v$ converges to $(W(v))_v$ as a process on the continuous function space ${\cal C}^0(\R)$. \\
As a consequence, a maximum $\widehat v=\sqrt n \, \big (\widehat \theta_n -\theta^* \big )$ of $\widehat W_n(v)$ converges in distribution to the maximum of ${}^t v \, \Gamma^*_1 \, v  + {}^t v  \, K_1$, which is $\displaystyle \overline v:= -\frac 1 2 \, \big (\Gamma_1^* \big )^{-1} K_1 \egaleloi {\cal N} \Big ( 0\, , \,(\sigma_\xi^2 -1) \,  \big (\Gamma_1^* \big )^{-1}\Gamma_2^* \, \big (\Gamma_1^* \big )^{-1}\Big)$ and this implies \eqref{CLT}.
\end{proof}
\begin{proof}[Proof of Remark \ref{Slu}] Using the notations of the proof of Theorem \ref{ANLARCH}, we have:
$$
\widehat \Gamma_1=\frac 1 n \, \sum_{s=1}^n \partial_\theta \widetilde M_{\widehat \theta_n}(s)\, {}^t \partial_\theta \widetilde M_{\widehat \theta_n}(s)\quad \mbox{and}\quad \widehat \Gamma_2:=\frac 1 n \, \sum_{s=1}^n \big (M_{\widehat \theta_n}(s) \big )^2 \partial_\theta \widetilde M_{\widehat \theta_n}(s)\, {}^t \partial_\theta \widetilde M_{\widehat \theta_n}(s).
$$
Using the same computations than those required for establishing \eqref{dMdM}, we have for any $\theta \in \Theta$:
$$
\big \| \partial _\theta M_{\theta}(s) \times {}^t \partial _\theta M_{\theta}(s)- \partial _\theta \widetilde M_{\theta}(s) \times {}^t \partial _\theta \widetilde  M_{\theta}(s) \big \|_2 
 \leq   C'' \,s^{\overline d-1/2}.
$$
And following \eqref{titi}, we obtain for any $\theta \in \Theta$
$$
\frac 1 n \, \sum_{s=1}^n \big ( \partial _\theta  M_{\theta}(s)\times {}^t \partial _\theta  M_{\theta}(s)- \partial _\theta \widetilde  M_{\theta}(s)\times {}^t \partial _\theta \widetilde  M_{\theta}(s) \big ) \limiteasn 0
$$
Using Theorem 2.2.1. in Straumann (2005), we deduce that:
$$
\sup_{\theta \in \Theta} \Big | \frac 1 n \, \sum_{s=1}^n \big ( \partial _\theta  M_{\theta}(s)\times {}^t \partial _\theta  M_{\theta}(s)- \partial _\theta \widetilde  M_{\theta}(s)\times {}^t \partial _\theta \widetilde  M_{\theta}(s) \big )\Big | \limiteasn 0
$$
and this implies 
\begin{equation}\label{gam1proof}
\Big |\frac 1 n \, \sum_{s=1}^n \big ( \partial _\theta  M_{\widehat \theta_n}(s)\times {}^t \partial _\theta  M_{\widehat \theta_n}(s)- \partial _\theta \widetilde  M_{\widehat \theta_n}(s) \times{}^t \partial _\theta \widetilde  M_{\widehat \theta_n}(s) \big )\Big | \limiteasn 0.
\end{equation}
Moreover, for any $\theta \in \Theta$, since the process $\big ( \partial _\theta  M_{\theta}(s)\times {}^t \partial _\theta  M_{\theta}(s)\big )_s$ is a stationary causal sequence such as $\E \big [\big \|\partial _\theta  M_{\theta}(0)\times {}^t \partial _\theta  M_{\theta}(0) \big \| \big ]<\infty$, Theorem 36.4 in Billingsley (1995)  implies :
$$
\frac 1 n \, \sum_{s=1}^n  \partial _\theta  M_{\theta^*}(s) \times{}^t \partial _\theta  M_{\theta^*}(s)\limiteasn \Gamma_1^*.
$$
But for any $s \in \{1,\ldots,n\}$, $\theta\in \Theta \mapsto  \partial _\theta  M_{\theta}(s)$ is a continuous function and since $\widehat \theta_n \limiteasn \theta^*$,
\begin{equation}\label{asn}
\frac 1 n \, \sum_{s=1}^n  \partial _\theta  M_{\widehat \theta_n}(s)\times {}^t \partial _\theta  M_{\widehat \theta_n}(s)\limiteasn \Gamma_1^*.
\end{equation}
Using \eqref{gam1proof} and \eqref{asn}, we deduce that $\widehat \Gamma_1 \limiteasn \Gamma_1^*$. \\
~\\
Using the same reasoning as well as \eqref{MdM}, we also deduce that $\widehat \Gamma_2 \limiteasn \Gamma_2^*$.
\end{proof}
\begin{proof}[Proof of Corollary \ref{corGLARCH}]
Using exactly the same arguments as in the proof of Corollary \ref{consGLARCH} for establishing that Assumption $C_{\ell}(\Theta)$ holds with $\ell=p+q+1$, Assumption $C_{\ell+2}(\Theta)$ also holds. \\
Moreover, for any GLARCH$(p,q)$ process, the matrix $\Gamma_1^*$ and $\Gamma_2^*$ are positive definite matrix. Indeed, following the same reasoning as in the proof of Lemma 5 of Beran and Sch\"utzner (2009), we have for any $v \in \R^{p+q+1}$
\begin{eqnarray*}
{}^t v \, \Gamma_1^* \, v =\E \Big [\big ({}^t v \,\partial_\theta M_{\theta^*}(0) \big )^2   \Big ] \geq 0.
\end{eqnarray*}
Assume that ${}^t v \,\partial_\theta M_{\theta^*}(0)=0$. By stationarity, this implies ${}^t v \,\partial_\theta M_{\theta^*}(k)=0$ for any $k\in \Z$. Using relation \eqref{GLARCH}, we deduce that:
$$
\partial_\theta M_{\theta^*}(0)=\partial_\theta \big ( c_0+c_1\, X_{-1}+\cdots +c_p\, X_{t-p} \big )+ \partial_\theta\big ( d_1 \, M_{\theta^*}(-1)+\cdots + d_q \, M_{\theta^*}(-q) \big ).
$$                           
Then:
$$
\left \{ \begin{array}{ccc}
\partial_{c_0} M_{\theta^*}(0)&=&1 + d_1 \, \partial_{c_0} M_{\theta^*}(-1)+\cdots + d_q \, \partial_{c_0} M_{\theta^*}(-q)  \\
\partial_{c_i} M_{\theta^*}(0)&=&X_{-i} + d_1 \, \partial_{c_i} M_{\theta^*}(-1)+\cdots + d_q \, \partial_{c_i} M_{\theta^*}(-q) \quad 1\leq i \leq p  \\
\partial_{d_j} M_{\theta^*}(0)&=& M_{\theta^*}(-j) + d_1 \, \partial_{d_j} M_{\theta^*}(-1)+\cdots + d_q \, \partial_{d_j} M_{\theta^*}(-q) \quad 1\leq j \leq q
\end{array} \right .
$$
Therefore, if ${}^t v \,\partial_\theta M_{\theta^*}(0)=0$, then $ \big ( 1,  X_{-1}, \ldots,X_{-p},M_{\theta^*}(-1),\ldots,M_{\theta^*}(-q) \big )\, v=0$. Such equation has no solution since a linear relationship \eqref{GLARCH} relies all these random variables to $M_{\theta^*}(0)$ (or these means that $(X_t)$ would be a GLARCH$(p-1,q-1)$ process which is not possible since we have assumed that $P_{\theta_1^*}$ and $Q_{\theta_2^*}$ are coprime  polynomials). Hence, ${}^t v \,\partial_\theta M_{\theta^*}(0)=0$ implies $v=0$: $\Gamma_1^*$ is a positive definite matrix (and a similar reasoning leads to the same property satisfied by $\Gamma_2^*$).
\end{proof}
\begin{proof}[Proof of Corollary \ref{corLARCH}]
As we had already written, in the case of a LARCH$(p)$ process, we must have $\Theta \subset \Theta_p(4)$, which is defined in  \eqref{Thetap}. Therefore, choosing $\Theta$ as defined in \eqref{Thep} guarantees that it is a compact subset of $\Theta_p(4)$. Moreover, since a LARCH$(p)$ process is a particular case of a GLARCH$(p,q)$ process, Corollary \ref{corGLARCH} is satisfied under the conditions of Corollary \ref{corLARCH}. 
\end{proof}

\begin{proof}[Proof of Corollary \ref{corLRD}]
In the proof of Corollary \ref{consLRD}, it was established that \eqref{CondAN} could be satisfied under the condition \eqref{LRD}, {\it i.e.} there exists $d(\theta)\in (0,1/2)$ and $L_\theta(\cdot)$ a slowly varying function such that $a_j(\theta)=L_\theta(j)\, j^{d(\theta)-1}$ for $j \in \N^*$. The various assumptions required for the establishment of the central limit theorem \eqref{CLT} are also present in the assumptions of the Corollary \ref{corLRD}. And condition \eqref{LRD} is not in contradiction when Assumption $C_{\ell+2}(\Theta)$ holds.
\end{proof}
\begin{proof}[Proof of Corollary \ref{corBeran2}]
In the proof of Corollary \ref{corBeran}, Assumption $C_3(\Theta)$ holds. But for $k\geq 3$, we also obtain $
\sup_{\theta \in \Theta} \big \|\partial^j _{\theta^j} a_k(\theta)\big \|\leq  \sqrt{c_M^2+j}\times \log^4 k \times k^{d_M-1} $ for $j=4,\, 5$. 
Then always with $\overline d =d_M/2+1/2$, Assumption $C_5(\Theta)$ holds. \\
Moreover, {\em mutatis mutendis}, we can use again the proof of Lemma 5 of Beran and Sch\"utzner (2009) for proving that $\Gamma_1^*$ and $\Gamma_2^*$ are two positive definite matrix in this case of long memory LARCH$(\infty)$ process.
\end{proof}
\paragraph*{Aknowledgments} The author is very grateful to the referees and the editor for many relevant suggestions, corrections and comments that helped to notably improve the  paper. 
\section*{Data Availability Statement}
Data available from request from the author.\\ Used softwares are available on {\tt https://samm.univ-paris1.fr/Sofwares-Logiciels}

%\bibliography{bib}

\end{document}